\documentclass[11pt]{article}
\usepackage{amssymb,amsmath}
\usepackage{amsthm,,latexsym}
\usepackage{graphicx}
\usepackage{enumerate}
\usepackage{mathdots}
\usepackage{appendix}
\usepackage{mathrsfs,amsfonts,dsfont}

\numberwithin{equation}{section}

\newtheorem{lemma}{LEMMA}[section]

\newtheorem*{thm1}{THEOREM 1}
\newtheorem*{thm2}{THEOREM 2}
\newtheorem*{thm3}{THEOREM 3}
\newtheorem*{thm4}{THEOREM 4}
\newtheorem*{thm5}{THEOREM 5}
\newtheorem*{thm6}{THEOREM 6}
\newtheorem*{thm7}{THEOREM 7}
\newtheorem*{thm8}{THEOREM 8}
\newtheorem*{thm9}{THEOREM 9}
\newtheorem*{thm10}{THEOREM 10}
\newtheorem*{corollary}{COROLLARY}
\newtheorem*{corollary1}{COROLLARY 1}
\newtheorem*{corollary2}{COROLLARY 2}

\theoremstyle{definition}

\newtheorem*{example}{Example}
\newtheorem*{rmk}{Remark}

\newtheorem*{defn}{Definition}

\newcommand{\ba}[1]{\begin{array}{@{}#1@{}}}
\newcommand{\ea}{\end{array}}
\long\def\symbolfootnote[#1]#2{\begingroup%
\def\thefootnote{\fnsymbol{footnote}}\footnote[#1]{#2}\endgroup}

\begin{document}
\newcommand{\ul}{\underline}
\newcommand{\be}{\begin{equation}}
\newcommand{\ee}{\end{equation}}
\newcommand{\ben}{\begin{enumerate}}
\newcommand{\een}{\end{enumerate}}

\long\def\symbolfootnote[#1]#2{\begingroup%
\def\thefootnote{\fnsymbol{footnote}}\footnote[#1]{#2}\endgroup}
\baselineskip = 16pt
\noindent

\title{On Higher Dimensional Fibonacci Numbers, Chebyshev Polynomials and Sequences of Vector Convergents}
\date{\today}

\author{M. W. Coffey (Mines), J. L. Hindmarsh, ${}^*$M. C. Lettington\\ and J. Pryce (Cardiff)$$}
\maketitle

\begin{abstract}
We study higher-dimensional interlacing Fibonacci sequences, generated via both Chebyshev type functions and $m$-dimensional recurrence relations. For each integer $m$, there exist both rational and integer versions of these sequences, where the underlying prime congruence structures of the rational sequence denominators enables the integer sequence to be recovered.

From either the rational or the integer sequences we construct sequences of vectors in $\mathbb{Q}^m$, which converge to irrational algebraic points in $\mathbb{R}^m$. The rational sequence terms can be expressed as simple recurrences, trigonometric sums, binomial polynomials, sums of squares, and as sums over ratios of powers of the signed diagonals of the regular unit $n$-gon. These sequences also exhibit a ``rainbow type'' quality, and correspond to the Fleck numbers at negative indices, leading to some combinatorial identities involving binomial coefficients.

It is shown that the families of orthogonal generating polynomials defining the recurrence relations employed, are divisible by the minimal polynomials of certain algebraic numbers, and the three-term recurrences and differential equations for these polynomials are derived. Further results relating to the Christoffel-Darboux formula, Rodrigues' formula and raising and lowering operators are also discussed. Moreover, it is shown that the Mellin transforms of these polynomials satisfy a functional equation of the form $p_n(s)=\pm p_n(1-s)$, and have zeros only on the critical line Re $s=1/2$.
\end{abstract}
\symbolfootnote[0]{2010 \emph{Mathematics Subject Classification}: 11B83, 11B39, 11J70, 33C45, 41A28.\newline
\emph{Key words and phrases}: Special Sequences and Polynomials, Generalised Fibonacci Numbers, Orthogonal Polynomials, Vector Convergents.}
\symbolfootnote[0]{${}^*$M. C. Lettington, School of Mathematics, Cardiff University, Wales, UK, CF24 4AG.
Email: LettingtonMC@cf.ac.uk Tel: +44 (0)2920 875670. The author would like to thank Prof. M. N. Huxley for his helpful comments.}

\section{Introduction}
Let $\mathcal{F}_0=1$, $\mathcal{F}_1=1$, $\mathcal{L}_0=2$ and $\mathcal{L}_1=1$ and define the $n$th Fibonacci number, $\mathcal{F}_n$, and the $n$th Lucas number, $\mathcal{L}_n$, in the usual fashion \cite{vajda}, so that $\mathcal{F}_{n+2}=\mathcal{F}_{n+1}+\mathcal{F}_n$, and the same recurrence for the Lucas numbers. We begin by examining the relationship that exists between the Fibonacci and Lucas numbers, binomial coefficients and the Chebyshev polynomials of the first and second kinds, which we now define.

\begin{defn}[of Chebyshev polynomials]
For $0\leq \theta \leq \pi$, and $n$ a non-negative integer, the Chebyshev polynomials (see for example \cite{mason,rivlin}) of the first and second kinds,
$T_n(x)$ and $U_n(x)$, are defined by
\be
T_n(\cos{\theta})=\cos{n\theta},\qquad U_n(\cos{\theta})=\frac{\sin{(n+1)\theta}}{\sin{\theta}}.
\label{eq:i1}
\ee
The related functions $\mathcal{C}_n(x)$ and $S_n(x)$, are then defined to be
$\mathcal{C}_n(x)=2T_n(x/2)$, $S_n(x)=U_n(x/2)$.
\end{defn}
Some key properties of the Chebyshev polynomials are given in Lemma 1.1, with equivalent Fibonacci and Lucas number expressions given in Lemma 1.2.
\begin{lemma}[Chebyshev identity lemma]
With $T_n(x)$ the Chebyshev polynomial of the first kind, $U_n(x)$ the Chebyshev polynomial of the second kind, as defined in $(\ref{eq:i1})$, we have
\be
T_n(x)=2^{n-1}\prod_{k=1}^n \left(x-\cos\left({{(2k-1)\pi } \over {2n}}\right)\right),
\label{eq:L1}
\ee
\be
U_n(x)=2^{n}\prod_{k=1}^n \left(x-\cos\left({{k\pi } \over {n+1}}\right)\right),
\label{eq:L2}
\ee
\be
U_n(x)=2\sum_{j ~\text{\rm\scriptsize{odd}}}^n T_j(x), ~~~n ~\mbox{\rm odd},\qquad U_n(x)=2\sum_{j ~\mbox{\rm\scriptsize even}}^n T_j(x)-1, ~~~n ~\mbox{\rm even},
\label{eq:L3}
\ee
\be
U_{n}(x)=\sum_{r=0}^{[n/2]}(-1)^r\binom{n-r}{r}(2x)^{n-2r}=\sum_{r=0}^{[n/2]} \binom{n+1}{2r+1}x^{n-2r}(x^2-1)^r,
\label{eq:L4}
\ee
and
\be
T_{n+1}(x)=2x T_n(x)-T_{n-1}(x),\qquad 2T_{m}(x)T_n(x)=T_{m+n}(x)+T_{|m-n|}(x).
\label{eq:L5}
\ee
\end{lemma}
\begin{proof}
For proofs of the above identities we refer the reader to chapters one and two of \cite{rivlin}.
\end{proof}

\begin{lemma}[Fibonacci identity lemma]
With $i^2=-1$, the Fibonacci and Lucas numbers can be expressed in terms of Chebyshev polynomials by
\be
\mathcal{F}_{n+1}=\frac{S_n(i)}{i^n},\qquad \mathcal{L}_{n}=\frac{\mathcal{C}_n(i)}{i^n},\qquad n=0,1,2,\ldots,
\label{eq:i2}
\ee
and in terms of binomial coefficients such that
\be
\mathcal{F}_{n+1}=\sum_{k=0}^{[n/2]}{\binom{n-k}{k}},\qquad \mathcal{L}_n=
\sum_{k=0}^{[n/2]}\frac{n}{n-k}\binom{n-k}{k}.
\label{eq:i0}
\ee
The binomial sum for the Fibonacci numbers is often referred to as the shallow diagonal sum of the Pascal triangle. The second sum can be viewed as a scaling of the terms in the previous sum, highlighting the interconnectedness that exists between the Fibonacci and Lucas sequences.

\end{lemma}
\begin{proof}
 For proofs of these Chebyshev and Fibonacci identities, we refer the reader to pages 60-64 of \cite{rivlin}, where it is shown that one can deduce the binomial sums in $(\ref{eq:i0})$ from the Chebyshev expressions in $(\ref{eq:i2})$.
\end{proof}

\begin{defn}[of cosine functions]
For $k\in \mathbb{Z}$, let $\phi_{m\, k}$, $\psi_{m\, k}$, $\mu_{m\,k}$ and $\nu_{m\,k}$ be defined such that
\[
\phi_{m\, k}=2\cos{\textstyle{\left (\frac{2\pi\,k}{2m+1}\right )}},\qquad \psi_{m\, k}=2\cos{\textstyle{\left (\frac{(2k-1)\pi}{2m+1}\right )}},
\]
\[
\mu_{m\,k}=\phi_{m\,k}-2,\qquad\nu_{m\,k}=\psi_{m\,k}-2,
\]
and
\[
e(x) = e^{2\pi ix} =\cos{(2\pi x)}+i\sin{(2\pi x)},
\]
so that $e(k/n)$ is an $n$th root of unity. Then working $\pmod{2m+1}$ the second subscript, we have
$\phi_{m\, j}\phi_{m\, k}=\phi_{m\, (j+k)}+\phi_{m\, (j-k)}$, as well as $-\psi_{m\,(k+m+1)}=\phi_{m\, k}$, and in terms of the $m$th roots of unity
\[
\phi_{m\,k}=e\textstyle{\left (\frac{k}{2m+1}\right )+ e\left (\frac{-k}{2m+1}\right )},\qquad
\psi_{m\,k}=e\textstyle{\left (\frac{2k-1}{4m+2}\right )+ e\left (\frac{-2k+1}{4m+2}\right )}.
\]
Hence $-\phi_{2\, 2}=(1+\sqrt{5})/2=\phi$,
is the Golden Ratio, from which it can be seen that $-\phi_{2\, 1}=(1-\sqrt{5})/2=\bar{\phi}=1/\phi_{2\, 2}$.
\end{defn}

\begin{defn}[of vector convergents]
We say that the sequence of vectors $\left \{{\bf v}_r\right \}_{r=1}^\infty$ in $\mathbb{Q}^m$ converges to the limit ${\bf v}\in \mathbb{R}^m$, if for any $\epsilon>0$, there exists an $r_{\epsilon}\in \mathbb{N}$, such that $|{\bf v}_r - {\bf v}|<\epsilon$ for $r>r_{\epsilon}$, where $|\,\,\,|$ denotes the standard Euclidean distance in $\mathbb{R}^m$.
\end{defn}
In this paper we identify the two (countably) infinite rational interlacing Fibonacci and Lucas sequences $\mathscr{F}_r^{(2,1)}$ and $\mathscr{F}_r^{(2,2)}$, for $r=1,2,3,\ldots$, and their multi-dimensional analogues $\mathscr{F}_r^{(m,j)}$, with $1\leq j\leq m$. From ratios of these sequence terms we construct recursively (countably) infinite sequences of vectors ${\bf \Psi}_r^{(m)}\in\mathbb{Q}^m$, which converge (as per the above definition) to the irrational algebraic points ${\bf \Phi}^{(m)}=(\phi_{m\,1},\ldots,\phi_{m\,m})\in\mathbb{R}^m$. These sequences of convergent vectors are a natural generalisation of the convergence of ratios of consecutive Fibonacci numbers $\mathcal{F}_{r+1}/\mathcal{F}_r$, to $\phi=(1+\sqrt{5})/2$, so that
\[
\mathop{\lim}_{r\rightarrow \infty}\frac{\mathcal{F}_{r+1}}{\mathcal{F}_r}=\phi.
\]
To give a structural overview of this work, in Section 2 we use a rational recurrence relation to construct the two interlacing Fibonacci and Lucas sequences

\begin{align}
\mathop{\mathscr{N}_r^{(2,1)}}_{r=1,2,3,\ldots} &=\mathcal{L}_0,-\mathcal{F}_1,\mathcal{L}_2,-\mathcal{F}_3,\mathcal{L}_4,-\mathcal{F}_5,\ldots = \left \{ \mathcal{L}_{2r}\right \}_{r=0}^\infty\cup \left \{ -\mathcal{F}_{2r+1}\right \}_{r=0}^\infty,\nonumber\\
\mathop{\mathscr{N}_r^{(2,2)}}_{r=1,2,3,\ldots} &=\mathcal{L}_1,-\mathcal{F}_2,\mathcal{L}_3,-\mathcal{F}_4,\mathcal{L}_5,-\mathcal{F}_6,\ldots =\left \{ \mathcal{L}_{2r+1}\right \}_{r=0}^\infty\cup \left \{ -\mathcal{F}_{2r}\right \}_{r=1}^\infty,\label{eq:int}
\end{align}
as the numerator sequences of the respective rational sequences $\mathscr{F}_r^{(2,1)}$ and $\mathscr{F}_r^{(2,2)}$. The divisibility and continued fraction properties of these sequences are examined, and closed form expressions analogous to Binet's formula are derived. The denominators of these rational sequence terms are powers of 5, thus indicating a 5-adic structure. To achieve alignment of the denominator factors between the two sequences, both sequences (as ordered above) commence with the term $r=1$.

Central to our results are the four families of polynomial functions $P_m(x),Q_m(x)$, $\mathcal{P}_m(x)$, and $\mathcal{Q}_m(x)$, which are introduced in Section 3, with connecting trigonometric, binomial, Fibonacci, Lucas and Chebyshev identities for these polynomials given in Theorem 1. The $m$-dimensional interlacing Fibonacci sequences of the title $\{\mathscr{F}_r^{(m,j)}\}_{r=1}^\infty$, with $1\leq j \leq m$, along with the sister sequences $\{\mathscr{G}_r^{(m,j)}\}_{r=1}^\infty$, are subsequently defined, and in Theorem 2 we obtain the generating functions for these sequences. As detailed in the Corollary to Theorem 2 and Lemma 3.2, this yields a number of ways to express $\mathscr{F}_r^{(m,j)}$ (and $\mathscr{G}_r^{(m,j)}$), such as
\[
\mathscr{F}_r^{(m,j)}
=\sum_{t=1}^{m}\left (\mu_{m\, t} \right )^{-r}\left (\phi_{m\, j t}-\phi_{m\,(j-1)t}\right )
=(-1)^{r-1}\sum _{t=1}^m
\textstyle{\frac{ \left(2 \sin \left(\frac{\pi  (2 j-1) t}{2m+1}\right)\right)}{\left(2 \sin \left(\frac{\pi  t}{2m+1}\right)\right)^{2 r-1}}}.
\]
With $n=2m+1$, the right-hand display allows for the geometric representation of the sequence terms as ratios of the diagonal lengths of an odd-sided regular $n$-gon inscribed in the unit circle (Theorem 5).

The convergence properties of the pairwise ratios of these sequences are then considered (Theorem~3), enabling the construction of the vector sequences ${\bf \Psi}_r^{(m)}~\in~\mathbb{Q}^m$, which converge to the limit point ${\bf \Phi}^{(m)}\in\mathbb{R}^m$. Bounds for the remainder term and connections with simple continued fractions are briefly discussed in the two corollaries. In Theorem 4  we show that the sequences $\mathscr{F}_r^{(m,j)}$ (and $\mathscr{G}_r^{(m,j)}$) are ``rainbow sequences'', consisting of
$n$ times a renumbering of the Fleck Numbers (alternating sums of binomial coefficients modulo $n$) for $r$ at negative integer values.
The non-reduced numerators $\mathscr{N}_{r}^{(m,j)}$
of $\mathscr{F}_{r}^{(m,j)}$ are described in Theorem 6, where in particular for $n=2m+1=p$ an odd prime number, it is shown that
\[
\left \{p^{\left \lfloor\frac{r-1}{m}\right \rfloor}\mathscr{F}_r^{(m,j)}\right \}_{r=-\infty}^{+\infty}=\left \{\mathscr{N}_r^{(m,j)}\right \}_{r=-\infty}^{+\infty},\,\,\,\text{\rm so that}\,\,\, p^{\left \lfloor\frac{r-1}{m}\right \rfloor}\mathscr{F}_r^{(m,j)}\in\mathbb{Z},\,\,\,\forall\,\,\,r\in\mathbb{Z}.
\]
In Theorem 7, we obtain sums of squares identities via the sequence term relations
\be
\sum_{j=1}^m\left ( \mathscr{F}_r^{(m,j)}\right )^2=-n\mathscr{F}_{2r}^{(m,1)},\qquad
\left (\mathscr{G}_r^{(m,0)}\right )^2+2\sum_{j=1}^m\left ( \mathscr{G}_r^{(m,j)}\right )^2=2n\mathscr{F}_{2r+1}^{(m,1)},
\label{eq:squ}
\ee
from which we deduce some combinatorial identities (not given in \cite{gould}), including
\be
\sum _{j=1}^m \left(\sum _{a=-\infty}^{\infty}\textstyle{ (-1)^{r+j+a} \binom{2 r+1}{r+j+a (2 m+1)}}\right)^2
=\sum _{a=-\infty}^{\infty}\textstyle{ (-1)^a \binom{4 r+1}{2 r+1+a (2 m+1)}}.
\label{eq:combin1}
\ee
In Section 4 we examine in greater detail our Fibonacci, Lucas and Chebyshev polynomial functions, deriving orthogonality conditions, differential and recurrence relations, Mellin transforms with zeros only on the critical line ${\mathcal{R}e(s)}=1/2$, a Christoffel-Darboux identity, and minimal polynomial relations. To conclude, in Section 5 we take a brief look at the sequences of matrix minors.

Aside from the binomial Fleck number representation given in Theorem 4, as far as the authors are aware, the results contained herein are seemingly new and unpublished.
Before deriving our results, we first reacquaint ourselves with the convergence properties of the Fibonacci and Lucas sequences.

\subsection*{Fibonacci Convergents}
The method known as Euclid's algorithm or the highest common factor rule or the continued fraction rule is central to classical number theory. This algorithm produces all ``good approximations'' to a given real number $\alpha$ which can be rephrased in terms of $2\times 2$ matrices such that, given a fraction $a/c$ in its lowest terms so that hcf$\,(a,c)=1$, find a matrix of integers
\[
\left (\ba{cc}
a & b\\
c & d
\ea\right )\qquad
\text{\rm with determinant}\qquad
\left |\ba{cc}
a & b\\
c & d
\ea\right |=ad-bc=1.
\]
As a consequence it means that much of number theory is fundamentally concerned with how pairs of integers behave, and so it is natural to study them accordingly.

Taking $a_0=[\alpha]$, the simple continued fraction algorithm produces a series of positive integers $a_1$, $a_2,\ldots$, from which one obtains the convergents $p_0/q_0, p_1/q_1, p_2/q_2,\ldots$, such that $p_0/q_0=a_0/1$, and thereafter
\[
\frac{p_1}{q_1}=a_0 + \cfrac{1}{a_1 },\qquad
\frac{p_2}{q_2}=a_0 + \cfrac{1}{a_1 + \cfrac{1}{a_2 }}\,,\,\ldots \ldots \,,\frac{p_n}{q_n}=a_0 + \cfrac{1}{a_1 + \cfrac{1}{a_2 + \ddots\cfrac{1}{a_n}}}\,.
\]
This algorithm allows the matrix representation
\be
\left (\ba{cc}p_{r+1}  &p_{r} \\		
	 q_{r+1}    &q_{r} \ea\right )
=\left (\ba{cc}p_r  &p_{r-1} \\		
	 q_r    &q_{r-1} \ea\right )
 \left (\ba{cc}a_{r+1}  &1 \\		
	 1    &0 \ea\right )
=\left (\ba{cc}a_{r+1}p_r+p_{r-1}  &p_r \\		
	 a_{r+1}q_r + q_{r-1}   &q_r \ea\right ),
\label{eq:mat1}
\ee
and by Dirichlet's theorem for continued fractions (see \cite{dirichlet} p131), the convergents have a remainder term satisfying
\be
\left|\frac{p_r}{q_r}-\alpha\right |\leq \frac{1}{q_r q_{r+1}}< \frac{1}{q_r^2}.
\label{eq:dirichlet}
\ee
As every real irrational number $\alpha$ has an infinite simple continued fraction expansion, one can think of $\mathbb{R}$ as being the completion of $\mathbb{Q}$ with respect to its Cauchy sequence limit points.

The continued fraction expansion for the Golden Ratio $\phi=(1+\sqrt{5})/2$ is $[1;\dot{1}]$, and so in terms of construction, it is the slowest converging simple continued fraction possible. From the expansion $[1;\dot{1}]$ we deduce that the sequence of convergents obey the relation $\frac{p_{r+1}}{q_{r+1}}=1+\frac{q_r}{p_r}$, with the first few terms in this sequence given below.
\[
\frac{1}{1},\frac{2}{1},\frac{3}{2},\frac{5}{3},\frac{8}{5},\frac{13}{8},\frac{21}{13},\frac{34}{21},\frac{55}{34},\ldots
\]
Hence
\be
\frac{p_{r}}{q_{r}}=\frac{\mathcal{F}_{r+2}}{\mathcal{F}_{r+1}}\quad\text{with}\quad \mathop{\lim}_{r\rightarrow \infty}\frac{\mathcal{F}_{r+2}}{\mathcal{F}_{r+1}}=\phi,
\quad\text{and}\quad \frac{p_{r+2}}{q_{r+2}}=\frac{p_{r+1}+p_r}{q_{r+1}+q_r},
\label{eq:mat2}
\ee
by $(\ref{eq:mat1})$,
concurring with the three term recurrence relation $\mathcal{F}_{r+2}=\mathcal{F}_{r+1}+\mathcal{F}_r$.
The corresponding recurrence equation is $x^2-x-1=0$, which has largest root $\phi$.
As the sequence of Lucas numbers \cite{vajda} obeys the same recurrence relation, with the initial values $\mathcal{L}_0=2$ and $\mathcal{L}_1=1$, we again find that the sequence of ratios of successive Lucas numbers converges to~$\phi$, and their reciprocals to $1/\phi=\phi_{2\,1}$.

In terms of $\phi_{2\,r}$, Binet's closed form expression for the Fibonacci sequence, and the Lucas sequence analogue (see Theorems 5.6 and 5.8 p79 of \cite{koshy}) equate to
\be
\mathcal{F}_n=\frac{(-1)^{n}}{\sqrt{5}}\left (\phi_{2\, 2}^n-\phi_{2\, 1}^{n}\right )=\sum_{r=1}^2\frac{(-1)^{k+n}}{\sqrt{5}} \phi_{2\, r}^n,
\label{eq:i355}
\ee
and
\be
\mathcal{L}_n=(-1)^n\left (\phi_{2\, 2}^n+\phi_{2\, 1}^{n}\right )=\sum_{r=1}^2 (-\phi_{2\, r})^n.
\label{eq:i3}
\ee
Subtracting the limit point of $\phi=-\phi_{2\,2}$ from the sequence ratios $\mathcal{F}_{n+1}/\mathcal{F}_n$; rewriting using the closed form (\ref{eq:i355}), and applying Dirichlet's theorem (\ref{eq:dirichlet}), we have that the Fibonacci convergent remainder terms satisfy
\[
\left |\frac{\mathcal{F}_{n+1}}{\mathcal{F}_n}-\phi\right | =\left |\frac{\sqrt{5}}{\phi^{2n}-(-1)^n}\right |<\frac{1}{\mathcal{F}_n^2}.
\]
Regarding the Lucas numbers for $n>3$, and using (\ref{eq:i3}), we obtain the slightly weaker bound
\[
\left |\frac{\mathcal{L}_{n+1}}{\mathcal{L}_n}-\phi\right | =\left |\frac{\sqrt{5}}{\phi^{2n}+(-1)^n}\right |<\frac{1}{\mathcal{L}_{n-2}^2}.
\]
Other variations on Binet's formula, connecting the fifth roots of unity and the Fibonacci sequence, are
\be
(1+\phi_{2\, k})^n=\mathcal{F}_{n+1}+\phi_{2\, k}\mathcal{F}_n,\qquad k\in \{1,2\},
\label{eq:i4}
\ee
discussed by Grzymkowski and Witula in \cite{gr}, and $(-\phi_{2\, 2})^n=-\phi_{2\, 2}\mathcal{F}_n+\mathcal{F}_{n-1}$, given by Vajda in \cite{vajda}.

From the relations given in $(\ref{eq:i2})$, $(\ref{eq:i3})$ and $(\ref{eq:i4})$ one can deduce a multitude of identities, of which some of the better known are
\begin{align}
&\mathcal{F}_{m+n+1}=\mathcal{F}_{m+1}\mathcal{F}_{n+1}+\mathcal{F}_m \mathcal{F}_n,\qquad&\sum_{k=1}^n \mathcal{F}_k^2=&\mathcal{F}_n \mathcal{F}_{n+1},\label{eq:fibo1}\\
&\mathcal{F}_{n+1}\mathcal{F}_{n-1}-\mathcal{F}_n^2=(-1)^n,\qquad&\mathcal{L}_{n+1}\mathcal{L}_{n-1}&-\mathcal{L}_n^2=(-1)^{n-1}5,\label{eq:fibo2}\\
&\mathcal{F}_{m+n}=\mathcal{L}_{m}\mathcal{F}_{n+1}-\mathcal{F}_{m-1}\mathcal{L}_n,\qquad & \mathcal{L}_{m+n}=5&\mathcal{F}_{m}\mathcal{F}_{n+1}-\mathcal{L}_{m-1}\mathcal{L}_{n}.\label{eq:fibo3}
\end{align}
The two identities in (\ref{eq:combin1}) (proved in Theorem 7) are a variation on the right-hand display in (\ref{eq:fibo1}), leading to representations of the sequence terms $\mathscr{F}_{2r}^{(m,1)}$ as a sum of the squares of $m$ non-zero integers, and the sequence terms $\mathscr{F}_{2r+1}^{(m,1)}$ as a sum of the squares of $2m+1$ positive integers.
The relations (\ref{eq:fibo2}) and (\ref{eq:fibo3}) highlight the interconnectedness that exists between these two sequences and also to the number~5. As a motivation for our general theories, we now describe an alternative recurrence approach that yields two rational sequences, whose numerators are interlacing Fibonacci and Lucas numbers, and denominators powers of 5.

\section{Interlacing Fibonacci and Lucas Sequences}
Let $\mathscr{F}_1^{(2,1)}=2$, $\mathscr{F}_2^{(2,1)}=-1$, $\mathscr{F}_1^{(2,2)}=1$, $\mathscr{F}_2^{(2,2)}=-1$, and thereafter
\[
\mathscr{F}_{r+2}^{(2,j)}=-\mathscr{F}_{r+1}^{(2,j)}-\frac{1}{5}\mathscr{F}_{r}^{(2,j)},\qquad j\in\{1,2\},\qquad r=0,1,2,3,\ldots
\]
so that in matrix notation
\[
\left(
\begin{array}{cc}
 \mathscr{F}_{r+1}^{(2,1)} & \mathscr{F}_r^{(2,1)} \\
 \mathscr{F}_{r+1}^{(2,2)} & \mathscr{F}_r^{(2,2)} \\
\end{array}
\right)
=
\left(
\begin{array}{cc}
 -1 & 2 \\
 1 & 0 \\
\end{array}
\right)
\left(
\begin{array}{cc}
 -1 & 1 \\
 -\frac{1}{5} & 0 \\
\end{array}
\right)^r,
\]
where $\mathscr{F}_r^{(2,1)}$ is the $r$th term of the first sequence and $\mathscr{F}_r^{(2,2)}$ the $r$th term of the second sequence. Then both sequences
satisfy a three-term recurrence relation, with recurrence equation $x^2+x+1/5=0$, whose roots are
\[
x=\frac{1-\sqrt{5}}{2\sqrt{5}}=\frac{-1}{\phi\sqrt{5}},\qquad \text{and}\qquad x=\frac{-1-\sqrt{5}}{2\sqrt{5}}=\frac{-\phi}{\sqrt{5}}.
\]

\begin{lemma}
The generating functions for the sequences $\mathscr{F}_r^{(2,1)}$ and $\mathscr{F}_r^{(2,2)}$ are given by
\[
\mathscr{F}_r^{(2,1)}:\,\,\,\frac{5 (x+2)}{x^2+5 x+5}=
2-x+\frac{3 x^2}{5}-\frac{2 x^3}{5}+\frac{7 x^4}{25}-\frac{5x^5}{25}+\frac{18 x^6}{125}-\frac{13
   x^7}{125}+\frac{47 x^8}{625}-\frac{34 x^9}{625}+\ldots,
\]
\be
\mathscr{F}_r^{(2,2)}:\,\,\,\frac{5x}{x^2+5 x+5}=
1-x+\frac{4 x^2}{5}-\frac{3 x^3}{5}+\frac{11 x^4}{25}-\frac{8 x^5}{25}+\frac{29 x^6}{125}-\frac{21
   x^7}{125}+\frac{76 x^8}{625}-\frac{55 x^9}{625}+\ldots,
\label{eq:gen}
\ee
so that $\mathscr{F}_r^{(2,1)}$ can be expressed in terms of $\mathscr{F}_r^{(2,2)}$ as
\[
\qquad \mathscr{F}_r^{(2,1)}=2 \mathscr{F}_{r}^{(2,2)}+\mathscr{F}_{r-1}^{(2,2)}.
\]
\end{lemma}
\begin{proof}
The result follows by applying the method of summation over the recurrence relation terms and rearranging. The relationship between the two sequences can then easily be seen by comparing the generating function structures.
\end{proof}

\begin{lemma} For $r=1,2,3,\ldots$ we have
\[
\mathscr{F}_r^{(2,1)}= \left (\frac{\phi_{2\, 2}}{\sqrt{5}}\right )^{r-1}+\left (-\frac{\phi_{2\, 1}}{\sqrt{5}}\right )^{r-1},\quad
\mathscr{F}_r^{(2,2)}=-\sqrt{5}\left (\left (\frac{\phi_{2\, 2}}{\sqrt{5}}\right )^r-\left (-\frac{\phi_{2\, 1}}{\sqrt{5}}\right )^r\right ),
\]
so that there exist closed form expressions for the sequences $\mathscr{F}_r^{(2,1)}$ and $\mathscr{F}_r^{(2,2)}$ in terms of $\phi_{2\, 1}$ and $\phi_{2\, 2}$.
\end{lemma}
\begin{corollary}
The sequence of ratios of consecutive terms $\mathscr{F}_{r+1}^{(2,j)}/\mathscr{F}_r^{(2,j)}$, with $j\in\{1,2\}$, approximates $-\phi/\sqrt{5}$ with an accuracy
\[
\left |\frac{\mathscr{F}_{r+1}^{(2,1)}}{\mathscr{F}_r^{(2,1)}}+\frac{\phi}{\sqrt{5}}\right | =\left |\frac{1}{\phi^{2r-2}+1}\right |,\qquad
\left |\frac{\mathscr{F}_{r+1}^{(2,2)}}{\mathscr{F}_r^{(2,2)}}+\frac{\phi}{\sqrt{5}}\right | =\left |\frac{1}{\phi^{2r}-1}\right |.
 \]
\end{corollary}
\begin{proof}
We give the proof for $\mathscr{F}_r^{(2,2)}$.  Applying Cauchy's Residue Theorem to the generating function (\ref{eq:gen}) of $\mathscr{F}_r^{(2,2)}$ yields
\[
\oint\,\,\,\frac{f(z)}{z^{r+1}}\,\, {\rm d}z= \oint\,\,\,\frac{5z}{z^{r+1}(z^2+5z+5)}\,\, {\rm d}z,
\]
where the contour contains the poles at $z=0$ and $z=\frac{-5\pm \sqrt{5}}{2}$, or equivalently $z=-\sqrt{5}\,\phi_{2\, 1}=\mu_{2\,1}$ and $z=\sqrt{5}\,\phi_{2\, 2}=\mu_{2\,2}$. The residue at 0 gives the term $\mathscr{F}_r^{(2,2)}$, and as the sum of the residues is 0, one obtains
\[
\mathscr{F}_r^{(2,2)}=-\frac{5}{\mu_{2\,1}^{r}(\mu_{2\,1}-\mu_{2\,2})}-\frac{5}{\mu_{2\,2}^{r}(\mu_{2\,2}-\mu_{2\,1})}.
\]
Using
\[
\mu_{2\,1}-\mu_{2\,2}=\sqrt{5}\quad\text{and}\quad \phi_{2\,2}=\frac{-1}{\phi_{2\, 1}},
\]
we deduce the desired closed form for $\mathscr{F}_r^{(2,2)}$, and similarly for $\mathscr{F}_r^{(2,1)}$.

The Corollary then follows from rearrangement of the closed form expressions of Lemma 2.2 in  conjunction with $-\phi/\sqrt{5}$ being the root of the recurrence equation with largest absolute value.
\end{proof}
\begin{lemma}[Numerator sequence lemma]
Let $\mathscr{N}_r^{(2,j)}$ be the non-reduced numerator of the $r$th term of the $j$th sequence $\mathscr{F}_r^{(2,j)}$, so that
\[
\mathscr{N}_r^{(2,j)}=5^{\lfloor\frac{r-1}{2}\rfloor} \mathscr{F}_r^{(2,j)},\qquad j\in\{1,2\}.
\]
with $\lfloor . \rfloor$ the floor function. Then for $r=1,2,3,\ldots$, we have
\[
\mathscr{N}_{2r-1}^{(2,1)}=\mathcal{L}_{2r-2},\quad \mathscr{N}_{2r}^{(2,1)}=\mathcal{F}_{2r-1},\quad\text{and}\quad \mathscr{N}_{2r-1}^{(2,2)}=\mathcal{L}_{2r-1},\quad \mathscr{N}_{2r}^{(2,2)}=\mathcal{F}_{2r}.
\]
\end{lemma}

\begin{proof}
Comparing the closed forms for $\mathscr{F}_r^{(2,1)}$ and $\mathscr{F}_r^{(2,2)}$ given in Lemma 2.2, with Binet's formula for the Fibonacci and Lucas numbers given in $(\ref{eq:i3})$, it then follows that the non-reduced numerators of the two sequences $\mathscr{F}_r^{(2,1)}$ and $\mathscr{F}_r^{(2,2)}$ are comprised of alternating Lucas and negative Fibonacci numbers (by non-reduced we mean that any common factors between numerator and denominator, such as in 5/25 in the sequence term $\mathscr{F}_5^{(2,1)}$, are not cancelled).
\end{proof}

As described in the Corollary to Lemma 2.2, the ratios of consecutive sequence terms yields two sequences of convergents with common limit point $-\phi/\sqrt{5}$. The initial terms of these two convergent sequences are given below.
\[
\begin{array}{ccccccccccc}
\displaystyle\frac{-\mathscr{F}_{r+1}^{(2,1)}}{\mathscr{F}_r^{(2,1)}}= \displaystyle\frac{1}{2}, & \displaystyle\frac{3}{5}, & \displaystyle\frac{2}{3}, & \displaystyle\frac{7}{10}, & \displaystyle\frac{5}{7}, & \displaystyle\frac{18}{25}, &
   \displaystyle\frac{13}{18}, & \displaystyle\frac{47}{65}, & \displaystyle\frac{34}{47}, & \displaystyle\frac{123}{170}, & \displaystyle\frac{89}{123},\ldots \\
\displaystyle\frac{-\mathscr{F}_{r+1}^{(2,2)}}{\mathscr{F}_r^{(2,2)}}= 1, & \displaystyle\frac{4}{5}, & \displaystyle\frac{3}{4}, & \displaystyle\frac{11}{15}, & \displaystyle\frac{8}{11}, & \displaystyle\frac{29}{40}, & \displaystyle\frac{21}{29}, &
   \displaystyle\frac{76}{105}, & \displaystyle\frac{55}{76}, & \displaystyle\frac{199}{275}, & \displaystyle\frac{144}{199},\ldots \\
\end{array}
\]
\begin{lemma}
Allowing negative integers $a_i$, recurring integer continued fraction expansions for the two sequences $\mathscr{F}_r^{(2,1)}$ and $\mathscr{F}_r^{(2,2)}$ are given by
\[
\mathscr{F}_r^{(2,1)}=[0; 1, 2, 3, -\dot{1}, \dot{5}],\qquad \mathscr{F}_r^{(2,2)}=[1;-\dot{5},\dot{1}].
\]
\end{lemma}
\begin{proof}
For $\mathscr{F}_r^{(2,1)}$ we assume that $[a_0,a_1,a_2,a_3]=[0,1,2,3]$, and then apply the continued algorithm to deduce that $a_4=a_6=-1$, and $a_5=a_7=5$, so that the expansion recurs from that point onwards. The continued fraction expansion for $\mathscr{F}_r^{(2,1)}$ can be deduced similarly.
\end{proof}
\begin{defn}[of divisibility sequences]
Let $a_1,a_2,a_3,\ldots$ be a sequence of integers satisfying the divisibility property
$s \mid t \Rightarrow a_s\mid a_t$, with $s$ and $t$ positive integers. Then we say that the sequence of integers $ \{a_r \}_{r=1}^\infty$ is a \emph{divisibility sequence}. Moreover, if we have the stronger condition
\[
\text{hcf}(s,r)=d\Rightarrow\text{hcf}(a_s,a_r)=a_d,
\]
then we say that it is a \emph{strong divisibility sequence}.
\end{defn}

\begin{lemma}
Let $s$ and $t$ be positive integers with $s \mid t$, so that $t/s$ is an integer. Then $\mathscr{N}_{s}^{(2,2)}\mid \mathscr{N}_{t}^{(2,2)}$, and the numerator sequence $ \{\mathscr{N}_{r}^{(2,2)} \}_{r=1}^\infty$ is a \emph{divisibility sequence}.
Moreover, we have that $\mathscr{N}_{2r}^{(2,2)}=\mathscr{N}_{r+1}^{(2,1)}\mathscr{N}_{r}^{(2,2)}$, and
if $s/d$ and $t/d$ are both odd integers, then $\text{hcf}(s,t)=d\Rightarrow\text{hcf}\left (\mathscr{N}_{s+1}^{(2,1)},\mathscr{N}_{t+1}^{(2,1)}\right )=|\mathscr{N}_{d+1}^{(2,1)}|$.

\end{lemma}
As a preamble to proving this lemma we now introduce the Gauss hypergeometric functions, as well as the Fibonacci and Lucas polynomials, which underpin many of the results discussed later on.

\begin{defn}
For integers $r,k,$ with $k\geq 0$ let the rising factorial, falling factorial and hypergeometric functions be respectively defined in the usual manner such that
\[
r^{\overline{k}}=r(r+1)\ldots (r+k-1),\qquad r^{\underline{k}}=r(r-1)\ldots (r-k+1),
\]
and
\[
{}_m\mathrm{F}_n\left ( \begin{array}{c|}a_1,\ldots,a_m\\
b_1,\ldots, b_n\end{array} \,\,z\right )=\sum_{k\geq 0} t_k,\qquad \mbox{where}\qquad
t_k=\frac{a_1^{\overline{k}}\ldots a_m^{\overline{k}}z^k}{b_1^{\overline{k}}\ldots b_n^{\overline{k}}k!},
\]
with none of the $b_i$ zero or a negative integer (to avoid division by zero).
\end{defn}

\begin{defn}[of Fibonacci and Lucas polynomials]
Let the Fibonacci polynomials $\mathcal{F}_n(x)$, be defined by the recurrence relation
\be
\mathcal{F}_{n+1}(x)=x \mathcal{F}_n(x) +\mathcal{F}_{n-1}(x), \qquad\text{with}\qquad \mathcal{F}_1(x)=1,\,\,\, \mathcal{F}_2(x)=x,
\label{eq:i17}
\ee
or equivalently by the explicit sum formula
\be
\mathcal{F}_n(x)=\sum_{j=0}^{[(n-1)/2]}\binom{n-j-1}{j}x^{n-2j-1}
=x^{n-1} \,
_2\mathrm{F}_1\left(\frac{1}{2}-\frac{n}{2},1-\frac{n}{2};1-n;-\frac{4}{x^2}\right).
\label{eq:175}
\ee
Similarly, define the Lucas polynomials $\mathcal{L}_n(x)$ by the recurrence relation in $(\ref{eq:i17})$, but with initial values $\mathcal{L}_1(x)=x$,
$\mathcal{L}_2(x)=x^2+2$, or equivalently by the explicit sum formula
\be
\mathcal{L}_n(x)=\sum _{j=0}^{\left\lfloor n/2\right\rfloor }\frac{n}{n-j}  \binom{n-j}{j}
x^{n-2 j}
=x^n \, _2\mathrm{F}_1\left(\frac{1}{2}-\frac{n}{2},-\frac{n}{2};1-n;-\frac{4}{x^2}\right).
\label{eq:176}
\ee
\end{defn}
\begin{lemma}[Factorisation and divisibility lemma]
The Fibonacci and Lucas polynomials and numbers have the following factorisation and divisibility properties.
\begin{itemize}
\item[{(1)}] With $i^2=-1$, the roots of the Fibonacci and Lucas polynomials can be expressed in terms of $i$ multiplied by the cosine of rational multiples of $\pi$, so that the polynomials can be factorised as
\be
\mathcal{F}_n(x)=\prod_{k=1}^{n-1}\left (x-2i\cos{\left (\frac{k\pi}{n}\right )}\right ),
\label{eq:i18}
\ee
and
\be
\mathcal{L}_n(x)=\prod_{k=0}^{n-1}\left (x-2i\cos{\left (\frac{(2k+1)\pi}{2n}\right )}\right ).
\label{eq:i19}
\ee
\item[{(2)}] The Fibonacci and Lucas polynomials have the divisibility properties
\be
\mathcal{F}_n(x)\mid \mathcal{F}_m(x) \Leftrightarrow n\mid m,\qquad \mathcal{F}_n\left (U_{p-1}\left (\sqrt{5}/2\right )\right )=\mathcal{F}_{np}/\mathcal{F}_p,
\label{eq:i184}
\ee
and
\be
\mathcal{L}_n(x)\mid \mathcal{L}_m(x) \Leftrightarrow m=(2k+1)n,\qquad \text{for some integer $k$}.
\label{eq:i185}
\ee
\item[{(3)}] The Fibonacci sequence is obtained by setting $x=1$, so that $\mathcal{F}_n=\mathcal{F}_n(1)$, and similarly for the Lucas sequence with $\mathcal{L}_n=\mathcal{L}_n(1)$. Taking $x\in\{2,3,4,\dots\}$, then produces one possible definition of higher dimensional Fibonacci and Lucas sequences which obey the divisibility properties as stated.
\item[{(4)}]  Two important divisibility characteristics of the Fibonacci sequence are that for $p$ a prime, $p$ divides $\mathcal{F}_{p-\left (\frac{5}{p}\right )}$, with $\left (\frac{5}{p}\right )$ the Legendre symbol, and that $\text{hcf}\,(s,r)=d\Rightarrow\text{hcf}\,(\mathcal{F}_s,\mathcal{F}_r)=\mathcal{F}_d$ so that it is a \emph{strong divisibility sequence}.
\item[{(5)}] With the extra constraint that $s/d$ and $t/d$ are both odd integers, an analogous divisibility sequence result holds for the Lucas numbers whereby $\text{hcf}\,(s,r)=d\Rightarrow\text{hcf}\,(\mathcal{L}_s,\mathcal{L}_t)=\mathcal{L}_d$. It follows that if $s/d$ is an odd integer, then $\mathcal{L}_d$ divides $\mathcal{L}_s$. The Lucas numbers also have the factoring property  $\mathcal{F}_{2r}=\mathcal{F}_r\mathcal{L}_r$, so that $\mathcal{L}_r\mid \mathcal{F}_{2r}$.
\end{itemize}
\end{lemma}
\begin{proof}
For proofs of these divisibility properties see T. Koshy \cite{koshy}, p196-214, and p451-479.
\end{proof}

\section{Higher-Dimensional Interlacing Fibonacci Sequences}

In order that we may generalise our previous results for $\mathscr{F}_r^{(2,1)}$ and $\mathscr{F}_r^{(2,2)}$, to $\mathscr{F}_r^{(m,j)}$ with $1\leq j\leq m$, we now introduce families of polynomials related to the Fibonacci and Lucas polynomials, that are central to our theories.

\begin{defn}[of generating function polynomials]
For positive integer $m$, we define $P_m(x)$, $Q_m(x)$, $\mathcal{P}_m(x)$, $\mathcal{Q}_m(x)$ and $V_m(x)$, to be the polynomials of degree $m$  given by
\be
P_m(x)=\sum_{k=0}^{m}\frac{2m+1}{2k+1}\binom{m+k}{2k}x^{k},\qquad
Q_m(x)=\sum_{k=0}^{m}\frac{m}{k}\binom{m+k-1}{2k-1}x^k,
\label{eq:s2}
\ee
\be
\mathcal{P}_m(x)=\sum_{k=0}^{m}\binom{m+k}{2k}x^{k},\qquad \mathcal{Q}_m(x)=\sum_{k=0}^{m}\binom{m+k+1}{2k+1}x^k,
\label{eq:s25}
\ee
and
\be
V_m(x)=\sum_{k=0}^m(-1)^{m+\left [ \frac{k}{2}+\frac{m}{2}\right ] } \binom{\left [
   \frac{k}{2}+\frac{m}{2}\right ] }{k}x^{k},
\label{eq:s255}
\ee
where the identity
\be
\mathop{\rm \lim}_{k\rightarrow 0}\,\,\frac{j}{k}\binom{j+k-1}{2k-1}=2,
\label{eq:s3}
\ee
ensures that $Q_m(x)$ is well defined.

We label the roots of $P_m(x)$, ordered in terms of increasing absolute value, by $\mu_{m\, 1}, \mu_{m\, 2},\ldots, \mu_{m\, m}$, and similarly $\nu_{m\, 1}, \nu_{m\, 2},\ldots, \nu_{m\, m}$ the ordered $m$ roots of $Q_m(x)$, so that we may write
\be
P_m(x)=\prod_{i=1}^m (x-\mu_{m\, i}),\quad Q_m(x)=\prod_{i=1}^m (x-\nu_{m\, i}),
\label{eq:s4}
\ee
where for $i <j$, we have $|\mu_{m\, i}|\leq |\mu_{m\, j}|$, and $|\nu_{m\, i}|\leq |\nu_{m\, j}|$.

\end{defn}
In Theorem 1 we show that the above definitions for $\mu_{m\, i}$ and $\nu_{m\, i}$ agree with those previously given in the definition following Lemma 1.2. This leads to simple identities for the polynomials $P_m(x)$, $Q_m(x)$, $\mathcal{P}_m(x)$, $\mathcal{Q}_m(x)$ and $V_m(x)$, in terms of Chebyshev $S_m(x)$ and $C_m(x)$ polynomials, as well as Fibonacci and Lucas polynomials.

\begin{thm1}
The roots of the equations $P_m(x)=0$ and $Q_m(x)=0$ are real, simple, negative, contained within the interval $[-4,0]$, and with the above definitions for $\mu_{m\,k}$ and $\nu_{m\,k}$, $ 1\leq k\leq m$, we have
\be
\mu_{m\,k}=\phi_{m\,k}-2=2\cos\left ( \frac{2\pi k}{2m+1}\right )-2,\qquad \nu_{m\,k}=\psi_{m\,k}-2=2\cos\left ( \frac{\pi(2 k-1)}{2m+1}\right )-2,
\label{eq:t15}
\ee
so that $ P_m(x)=$
\be
\prod_{k=1}^m \left (x+2-2\cos\left ( \frac{2\pi k}{2m+1}\right )\right )=\frac{1}{\sqrt{x}}\mathcal{L}_{2m+1}\left (\sqrt{x}\right )
 ={1 \over \sqrt{x}}\left (\mathcal{F}_{2m+2}(\sqrt{x})+\mathcal{F}_{2m}(\sqrt{x})\right),
\label{eq:s6}
\ee
\be
=U_{2m}\left(\sqrt{1+\frac{x}{4}}\right)=S_{2m}\left(\sqrt{x+4}\right)=S_{2m}\left(2\cos{y}\right),\quad\text{\rm with}\quad x=2\cos{2y}-2,
\label{eq:ep1}
\ee
\be
Q_m(x)=\prod_{k=1}^m \left (x+2-2\cos\left ( \frac{\pi (2k-1)}{2m}\right )\right )=\mathcal{L}_{2m}\left (\sqrt{x}\right )
=\mathcal{F}_{2m+1}(\sqrt{x})+\mathcal{F}_{2m-1}(\sqrt{x})
\label{eq:s7}
\ee
\be
=2 T_{2m}\left(\sqrt{1+\frac{x}{4}}\right)
=\mathcal{C}_{2m}\left(\sqrt{x+4}\right)=\mathcal{C}_{2m}\left(2\cos{y}\right),\quad\text{\rm with}\quad x=2\cos{2y}-2,
\label{eq:ep2}
\ee
and
\be
\mathcal{P}_m(x)=\prod_{k=1}^m \left (x+2+2\cos\left ( \frac{2\pi k}{2m+1}\right )\right )=\mathcal{F}_{2m+1}\left (\sqrt{x}\right ),
\label{eq:s71}
\ee
\be
=U_{2m}\left(\sqrt{\frac{-x}{4}}\right)=S_{2m}\left(\sqrt{-x}\right)=S_{2m}\left(2i\cos{y}\right),\quad\text{\rm with}\quad x=2\cos{2y}+2,
\label{eq:ep3}
\ee
\be
\mathcal{Q}_m(x)=\frac{1}{\sqrt{x}}\,\mathcal{F}_{2m+2}\left (\sqrt{x}\right ),
\label{eq:q1}
\ee
and
\be
V_m(x)=P_m(x-2)=\prod_{k=1}^m \left (x-2\cos\left ( \frac{2\pi k}{2m+1}\right )\right )
\label{eq:s72}
\ee
\be
=U_{2m}\left(\sqrt{{1+x/2} \over 2}\right)
=S_{2m}\left(\sqrt{x+2}\right)=S_{2m}\left(2\cos{y}\right),\quad\text{\rm with}\quad x=2\cos{2y}.
\label{eq:ep4}
\ee
\end{thm1}
\begin{corollary}
We also have the relations
\be
x P_{m-1}(x)=Q_{m}(x)-Q_{m-1}(x),\qquad Q_{m}(x)=P_{m}(x)-P_{m-1}(x),
\label{eq:lms31}
\ee
\be
x\mathcal{P}_{m-1}(x)=\mathcal{Q}_{m}(x)-\mathcal{Q}_{m-1}(x),\qquad  \mathcal{Q}_{m}(x)=\mathcal{P}_{m}(x)-\mathcal{P}_{m-1}(x).
\label{eq:lms32}
\ee
\be
\mathcal{P}_m(x)=(-1)^{m}P_m(-x-4),
\label{eq:s725}
\ee
\be
x\,P_m(-x^2)=(-1)^{m-1}\,Q_{2m+1}(-x-2),\qquad
Q_m(-x^2)=(-1)^{m}Q_{2m}(-x-2),
\label{eq:s7405}
\ee
and
\be
\mathcal{Q}_{m}(x)=
\begin{cases}
x\,P_{m_1}(x)\mathcal{P}_{m_1}(x) &\qquad \text{\rm if $m=2m_1+1$ is odd},\\
Q_{m_1+1}(x)\mathcal{Q}_{m_1}(x) &\qquad \text{\rm if $m=2m_1$ is even},
\end{cases}
\label{eq:s745}
\ee
along with the integral identity
\be
(2m+1)\int_0^{\sqrt{x}} {\cal P}_m(t^2)dt=\sqrt{x}P_m(x)\Rightarrow (2m+1){\cal P}_m(x)=P_m(x)+2x P_m'(x).
\label{eq:s741}
\ee
\end{corollary}

\begin{proof}[Proof of Theorem 1]
For $k=1,2,\ldots, m$, the function $\cos{\left (\frac{2 \pi k }{2m+1}\right )}$ is a decreasing function of $k$. It follows that $\phi_{m\, k}$ is a decreasing function of $k$, and so in terms of absolute values we have, $\mu_{m\, 1}<\mu_{m\, 2}<\ldots < \mu_{m\, m}$. A similar argument holds for the $\nu_{m \,k}$, $1\leq k \leq m$, and hence $(\ref{eq:s4})$.

Although it follows from $(\ref{eq:t15})$ that the roots of $P_m(x)$ are simple and lie in the interval $[-4,0]$, we demonstrate this by two other methods in order to highlight the links that exist between $P_m(x)$ and the Legendre and Chebyshev functions .

{\bf Method 1.}
By using manipulations with Pochhammer symbols which we omit, $P_m(x)$ may be written in terms of
the Gauss hypergeometric function $_2\mathrm{F}_1$.  Letting $P_n^m(x)$ denote the associated Legendre
function, in turn $P_m(x)$ may be expressed as
\be
P_m(x)={(2m+1)\sqrt{\pi} \over 2}{{(x+4)^{1/4}} \over {(-x)^{1/4}}} P_m^{-1/2}\left(1+{x \over 2}\right).
\label{eq:th81}
\ee
The functions $P_m^{-1/2}(z)$ are orthogonal on the interval $[-1,1]$.  With $x=2(z-1)$, it follows
from a standard result in the theory of orthogonal polynomials \cite{szego} that the zeros of
$P_m^{-1/2}\left(1+{x \over 2}\right)$ are contained in $[-4,0]$ and simple and hence for $P_m(x)$ too.  \qed

{\bf Method 2}  It is known that (e.g., \cite{magnus}, p. 64)
$$P_\nu^{-1/2}(\cos \varphi)=\sqrt{2 \over {\pi \sin \varphi}}{{\sin[(\nu+1/2)\varphi]} \over {(\nu+1/2)}}.$$
Then with $\varphi$ replaced by $\cos^{-1} \varphi$ and $x=2(\varphi-1)$ it again follows that the zeros of
$P_m(x)$ are in $[-4,0]$ and simple. \qed

{\it Remark}.  It is then possible to write
$$P_\nu^{-1/2}(\varphi)=\sqrt{2 \over \pi} {{(1-\varphi^2)^{1/4}} \over {(\nu+1/2)}}U_{\nu-1/2}(\varphi),$$
where $U_{\nu-1/2}$ is the Chebyshev function of the second kind.

To see $(\ref{eq:s72})$ and the first part of $(\ref{eq:s6})$, we have
$$
\prod_{k=1}^n \left(x-2\cos\left({{2\pi k} \over {2n+1}}\right)\right)=
2^n \prod_{k=1}^n \left({x \over 2}-\cos\left({{2\pi k} \over {2n+1}}\right)\right)
=1+2\sum_{k=1}^n T_k\left({x \over 2}\right)
$$
\be
 =U_n\left({x \over 2}\right)+U_{n-1}\left({x \over 2}\right)
=U_{2n}\left(\sqrt{{1+x/2} \over 2}\right)=P_n(x-2)=V_n(x),
\label{eq:pf11}
\ee
where we have used the relations $(\ref{eq:L2})$,  $(\ref{eq:L3})$ and  $(\ref{eq:L4})$, of Lemma 2.1.

Hence
$$\prod_{k=1}^n \left(x+2-2\cos\left({{2\pi k} \over {2n+1}}\right)\right)
=1+2\sum_{k=1}^n T_k\left({x \over 2}+1\right)$$
\be
= U_n\left({x \over 2}+1\right)+U_{n-1}\left({x \over 2}+1\right)
=U_{2n}\left(\sqrt{1 +{x \over 4}}\right)=P_n(x),
\label{eq:pf12}
\ee
and the result follows. Similar arguments can be used to obtain the first statements of $(\ref{eq:s7})$ and $(\ref{eq:s71})$.

The latter statements in  $(\ref{eq:s6})$, $(\ref{eq:s7})$, $(\ref{eq:s71})$, and that of  $(\ref{eq:q1})$, concerning expressions for $P_m(x),Q_m(x),\mathcal{P}_m(x)$ and $\mathcal{Q}_m(x)$ in terms of Fibonacci and Lucas polynomials, can be deduced via binomial relations as follows. We have
\[
\frac{1}{\sqrt{x}}\,\,\mathcal{L}_{2m+1}(\sqrt{x})=\frac{1}{\sqrt{x}}\sum_{j=0}^m \frac{2m+1}{2m+1-j}\binom{2m+1-j}{j}\left (\sqrt{x}\right )^{2m+1-2j}
\]
\[
=\sum_{j=0}^m \frac{2m+1}{2m+1-j}\binom{2m+1-j}{j}x^{m-j}=\sum_{j=0}^m \frac{2m+1}{2j+1}\binom{m+j}{m-j}x^{j}
\]
\[
=\sum_{j=0}^m \frac{2m+1}{2j+1}\binom{m+j}{2j}x^{j}=P_m(x),
\]
and
\[
\mathcal{L}_{2m}(\sqrt{x})=\sum_{j=0}^m \frac{2m}{2m-j}\binom{2m-j}{j}\left (\sqrt{x}\right )^{2m-2j}
\]
\[
=\sum_{j=0}^m \frac{2m}{2m-j}\binom{2m-j}{j}x^{m-j}=\sum_{j=0}^m \frac{m}{j}\binom{m+j-1}{2j-1}x^{j}=Q_m(x).
\]
The Fibonacci polynomial identities follow similarly.
The respective generating functions for the Lucas and Fibonacci polynomials (see page 447 of \cite{koshy}) are
$$G_L(x,t)=\sum_{n=0}^\infty \mathcal{L}_n(x)t^n={{1+t^2} \over {1-t^2-tx}}, ~~~~G_F(x,t)=\sum_{n=0}^\infty \mathcal{F}_n(x)t^n={t \over {1-t^2-tx}},$$
so that $tG_L(x,t)=(1+t^2)G_F(x,t)$. Hence we have $\mathcal{L}_n(x)=\mathcal{F}_{n+1}(x)+\mathcal{F}_{ n-1}(x),$
giving
$$P_m(x)={1 \over \sqrt{x}}\left[\mathcal{F}_{2m+2}(\sqrt{x})+\mathcal{F}_{2m}(\sqrt{x})\right],\qquad
Q_m(x)=\mathcal{F}_{2m+1}(\sqrt{x})+\mathcal{F}_{2m-1}(\sqrt{x}),$$
and so the ordinary generating functions of $P_m(x)$ and $Q_m(x)$ can be written as
\be
G_P(x,t)=\sum_{n=0}^\infty P_n(x)t^n={\textstyle{{{1+t} \over {1-(2+x)t+t^2}}}}, ~~~~~
  G_Q(x,t)=\sum_{n=0}^\infty Q_n(x)t^n={\textstyle{{{2-(2+x)t} \over {1-(2+x)t+t^2}}}}.
\label{eq:genp}
\ee
Combining the above relations with the identity $\mathcal{F}_{2m}(x)=\mathcal{F}_m(x)\mathcal{L}_m(x)$,
we obtain
\[
P_m(x)\mathcal{P}_m(x)=\frac{1}{\sqrt{x}}\mathcal{L}_{2m+1}(\sqrt{x})\mathcal{F}_{2m+1}(\sqrt{x})=\frac{1}{\sqrt{x}}\mathcal{F}_{4m+2}(\sqrt{x})=\mathcal{Q}_{2m}(x),
\]
\[
Q_{m+1}(x)\mathcal{Q}_m(x)=\frac{1}{\sqrt{x}}\mathcal{L}_{2m+2}(\sqrt{x})\mathcal{F}_{2m+2}(\sqrt{x})=\frac{1}{\sqrt{x}}\mathcal{F}_{4m+4}(\sqrt{x})=\mathcal{Q}_{2m+1}(x).
\]
The Chebyshev identities in $(\ref{eq:ep1})$, $(\ref{eq:ep2})$, $(\ref{eq:ep3})$ and $(\ref{eq:ep4})$ can be derived directly from the definition $(\ref{eq:i1})$, with further connections to the Chebyshev polynomials established via the expression for $P_m(x)$ given by
$$P_m(x)=T_m(1+x/2)+\sqrt{1+4/x}\sinh[2m ~\mbox{csch}^{-1}(2/\sqrt{x})].$$
The expressions (\ref{eq:lms31}) and (\ref{eq:lms32}) of the Corollary, follow from direct manipulation of the binomial coefficient polynomials
 (see Lemma 2.2 of \cite{mwcmcl1}) given in the definitions for $P_m(x)$, $Q_m(x)$, $\mathcal{P}_m(x)$, and $\mathcal{Q}_m(x)$. To see $(\ref{eq:s7405})$, substituting $-(x+2)$ in the product formula for $P_m(x)$ in $(\ref{eq:s6})$ and comparing with $(\ref{eq:s72})$, gives us $P_m(x-2)=V_m(x)$, and writing
\[
Q_{2m+1}(-(x+2))=\prod_{k=1}^{2m+1} \left (-x-2\cos\left ( \frac{\pi (2k-1)}{4m+2}\right )\right )
\]
\[
=\left (-x-2\cos\left ( \frac{\pi }{2}\right )\right )\prod_{k=1}^{m} \left (-x-2\cos\left ( \frac{\pi (2k-1)}{4m+2}\right )\right )
\left (-x+2\cos\left ( \frac{\pi (2k-1)}{4m+2}\right )\right )
\]
\[
=\left (-x\right )\prod_{k=1}^{m} \left (x^2-4\cos^2\left ( \frac{\pi (2k-1)}{4m+2}\right )\right )
=\left (-x\right )\prod_{k=1}^{m} \left (x^2-2\left (1+\cos\left ( \frac{\pi (2k-1)}{2m+1}\right )\right )\right )
\]
\[
=\left (-x\right )\prod_{k=1}^{m} \left (x^2-2-\psi_{m\, k}\right )
=\left (-1\right )^{m-1} x\prod_{k=1}^{m} \left (-x^2+2-\phi_{m\, k}\right )=(-1)^{m-1}x\,P_m(-x^2),
\]
we obtain the first expression in $(\ref{eq:s7405})$.  Similar arguments produce $(\ref{eq:s725})$, the second expression in $(\ref{eq:s7405})$, and $(\ref{eq:s745})$. Considering $(\ref{eq:s741})$ we have
\[
\int_{0}^{\sqrt{x}}\sum_{k=0}^m\, \binom{m+k}{2k}t^{2k}{\rm d}t=\left. t\sum_{k=0}^m \frac{\binom{m+k}{2k}t^{2k}}{2k+1}\right |_{t=0}^{\sqrt{x}}
=\sqrt{x}\,P_m(x),
\]
and differentiating we obtain the final display.

\end{proof}

\begin{defn}[of $m$-dimensional interlacing Fibonacci sequences]
Let the matrices of binomial coefficients $B_{\rm odd}$ and $B_{\rm even}$ be defined such that
\[
B_{\rm odd}=(b_{i, j})_{m\times m},\qquad\text{with} \qquad b_{i,j}=(-1)^{i+j-1} \binom{2 j-1}{j-i},
\]
and
\[
B_{\rm even}=(b_{i, j})_{m\times m},\qquad \text{with} \qquad b_{i,j}=(-1)^{i+j} \binom{2 j}{j-i},
\]
and let the recurrence matrix $R_m$ be given by
\be
R_m=\left(
\begin{array}{ccccc}
 -h_1 & 1 & 0 & \ldots  & 0 \\
 -h_2 & 0 & 1 & \ldots  & 0 \\
 \vdots & \vdots & \vdots &  \ddots & \vdots \\
 -h_{m-1} & 0 & 0  & 0 & 1\\
 -h_m & 0 & 0 & 0  & 0
\end{array}
\right),\qquad \text{with}\qquad
h_k = \frac{1}{2k+1}\binom{m+k}{2k},
\label{eq:fd}
\ee
Also let the sequences of $m\times m$ matrices $M_o(m,k)$ and $M_e(m,k)$ be respectively defined such that
\[
M_{o}(m,r)=(2m+1)B_{odd}\,R_m^{m+r},\qquad \text{and}\qquad M_e(m,r)=(2m+1)B_{even}\,R_m^{m+r},
\]
so that each row of $M_{o}(m,0)$ or $M_{e}(m,0)$ corresponds to a list of consecutive sequence values. The matrices $M_{o}(m,0)$ or $M_{e}(m,0)$ are then taken as the initial value matrices for the two sets of $m$ interlacing Fibonacci sequences that they generate, under repeated multiplication on the right, by the recurrence matrix $R_m$.

We denote by $\mathscr{F}_{r}^{(m,j)}$, $1\leq j \leq m$, the $r$th term in the $j$th \emph{rational interlacing Fibonacci sequence of dimension $m$},
and by $\mathscr{G}_{r}^{(m,j)}$, $1\leq j \leq m$, the type $\mathscr{G}$ $j$th \emph{rational interlacing Fibonacci sequence of dimension $m$}. The case $m=2$, then corresponds to the interlacing Fibonacci and Lucas sequences $\mathscr{F}_{r}^{(2,1)}$ and $\mathscr{F}_{r}^{(2,2)}$, defined in Section~2, and for the case $m=1$, we simply have
\[
\mathscr{F}_{r}^{(1,1)}=\left( \frac{-1}{3}\right )^r,\qquad r=1,2,3,\ldots.
\]
It follows that row $j$ of $M_{o}(m,r)$ contains the $j$th $m$-dimensional rational interlacing Fibonacci sequence terms $\mathscr{F}_{r+m}^{(m,j)},\mathscr{F}_{r+m-1}^{(m,j)},\ldots,\mathscr{F}_{r+1}^{(m,j)}$, and row $j$ of $M_{e}(m,r)$ contains the type $\mathscr{G}$,  $j$th $m$-dimensional rational interlacing Fibonacci sequence terms $\mathscr{G}_{r+m}^{(m,j)},\mathscr{G}_{r+m-1}^{(m,j)},\ldots,\mathscr{G}_{r+1}^{(m,j)}$, where both sequences satisfy the recurrence relation (given here in terms of $\mathscr{F}_r^{(j,m)}$)
\[
\mathscr{F}_{r+m}^{(m,j)}=-\frac{1}{3}\binom{m+1}{2} \mathscr{F}_{r+m-1}^{(m,j)}-\frac{1}{5}\binom{m+2}{4}\mathscr{F}_{r+m-2}^{(m,j)}-\ldots - \frac{1}{2m+1}\binom{m+m}{2m}\mathscr{F}_{r+m}^{(m,r)}.
\]
For example, when $m=5$ and $r=2$, we have
\[
M_o(5,2)=
(2m+1)\left(
\begin{array}{ccccc}
 -1 & 3 & -10 & 35 & -126 \\
 0 & -1 & 5 & -21 & 84 \\
 0 & 0 & -1 & 7 & -36 \\
 0 & 0 & 0 & -1 & 9 \\
 0 & 0 & 0 & 0 & -1 \\
\end{array}
\right)
\left(
\begin{array}{ccccc}
 -5 & 1 & 0 & 0 & 0 \\
 -7 & 0 & 1 & 0 & 0 \\
 -4 & 0 & 0 & 1 & 0 \\
 -1 & 0 & 0 & 0 & 1 \\
 -\frac{1}{11} & 0 & 0 & 0 & 0 \\
\end{array}
\right)^7
\]
\[
=
\left(
\begin{array}{ccccc}
 \mathscr{F}_7^{(5,1)} & \mathscr{F}_6^{(5,1)} & \ldots & \mathscr{F}_3^{(5,1)} \\
 \mathscr{F}_7^{(5,2)} & \mathscr{F}_6^{(5,2)} &  \ldots & \mathscr{F}_3^{(5,2)} \\
\mathscr{F}_7^{(5,3)} & \mathscr{F}_6^{(5,3)} &  \ldots & \mathscr{F}_3^{(5,3)} \\
\mathscr{F}_7^{(5,4)} & \mathscr{F}_6^{(5,4)} & \ldots& \mathscr{F}_3^{(5,4)} \\
\mathscr{F}_7^{(5,5)} & \mathscr{F}_6^{(5,5)} &  \ldots & \mathscr{F}_3^{(5,5)} \\
\end{array}
\right)
=\left(
\begin{array}{ccccc}
 \frac{10744}{11} & -\frac{3415}{11} & 99 & -32 & 11 \\
 \frac{28817}{11} & -\frac{9156}{11} & 265 & -85 & 28 \\
 \frac{37734}{11} & -\frac{11982}{11} & 346 & -110 & 35 \\
 \frac{34669}{11} & -\frac{11002}{11} & 317 & -100 & 31 \\
 \frac{20602}{11} & -\frac{6535}{11} & 188 & -59 & 18 \\
\end{array}
\right).
\]
\end{defn}
\begin{lemma}
The product of the eigenvalues of the recurrence matrix $R_m$ (and so its determinant) is given by $(-1)^m/(2m+1)$, and the sum of the eigenvalues by $-h_1=-m(m+1)/6$. The binomial matrices satisfy ${\rm Det(B_{odd})}=(-1)^m$ and ${\rm Det(B_{even})}=1$, so that the two sequences of determinants of the matrices $M_o(m,k)$ and $M_e(m,k)$,  for $k=0,1,2,3,\ldots$, consists of terms of the form $\pm \frac{1}{(2m+1)^k}$.

The inverse matrix of $R_m$ is given by
\[
R_m^{-1}=
\left(
\begin{array}{ccccc}
 0 & 0 & 0 & 0 & -(2m+1) \\
1 & 0 & 0 & 0 & -(2m+1)h_1 \\
 0 & 1 & 0 & 0 & -(2m+1)h_2 \\
 \vdots & \vdots & \ddots & \vdots & \vdots \\
 0 & 0 & 0 &  1 & -(2m+1)h_{m-1} \\
\end{array}
\right)
=
\left(
\begin{array}{ccccc}
 0 & 0 & 0 & 0 & -g_m \\
1 & 0 & 0 & 0 & -g_{m-1} \\
 0 & 1 & 0 & 0 & -g_{m-2} \\
 \vdots & \vdots & \ddots & \vdots & \vdots \\
 0 & 0 & 0 &  1 & -g_{1} \\
\end{array}
\right),
\]
where
\be
g_k = \frac{2m+1}{2m+1-2k}\binom{2m-k}{k},
\qquad\text{so that}\qquad g_k=(2m+1) h_{m-k}.
\label{eq:gd}
\ee
Denoting by $R_m(x)$ and $R^{-1}_m(x)$ the respective characteristic polynomials of $R_m$ and $R_m^{-1}$, we have
\[
R_m(x)=-\sum_{j=0}^m h_{m-j}\,x^j,\qquad R^{-1}_m(x)=-(2m+1)\sum_{j=0}^m h_{j}\,x^j=-P_m(x),
\]
where the eigenvalues of the inverse recurrence matrix $R_m^{-1}$ are the roots of the polynomial $P_m(x)$ and the eigenvalues of the recurrence matrix $R_m$ are the roots of the polynomial $P_m^{-1}(x)$ (say). By (\ref{eq:t15}) of Theorem 1, the roots of $P_m^{-1}(x)$ are given in descending order by $\mu_{m\,1}^{-1},\mu_{m\,2}^{-1},\ldots,\mu_{m\,m}^{-1}$. Hence, with a suitable choice of (algebraic number) coefficients $a_i^{(m,j)}$ we can write
\be
\mathscr{F}_r^{(m,j)}=a_1^{(m,j)}\mu_{m\,1}^{-r}+a_2^{(m,j)}\mu_{m\,2}^{-r}+\ldots a_m^{(m,j)}\mu_{m\,m}^{-r}.
\label{eq:ex1}
\ee
Moreover, if $\,2m+1$ is an odd prime number, then $R_m^m$ contains integers above the principal diagonal and rational numbers with denominator $(2m+1)$ in its lower triangular matrix part. Consequently, the initial value matrices $M_o(m,0)$ and $M_e(m,0)$ for the respective sequences $\mathscr{F}_r^{(m,j)}$ and $\mathscr{G}_r^{(m,j)}$ are integer matrices if and only if $2m+1$ is an odd prime number.
\end{lemma}
\begin{proof}
For a general $m\times m$ recurrence matrix $K_m$ of the form
\be
K_m=\left(
\begin{array}{cccccc}
 -a_1 & 1 & 0 & 0& \ldots & 0 \\
 -a_2 & 0 & 1 & 0 &\ldots  & 0 \\
 \vdots & \vdots & \vdots & \vdots & \ddots & \vdots \\
 -a_{m-1} & 0 & 0 & 0 &\ldots & 1 \\
 -a_m & 0 & 0 & 0 &\ldots & 0 \\
\end{array}
\right),
\label{eq:chr1}
\ee
evaluating the determinant along the first column yields $|K_m|=(-1)^{m}a_m$, and similarly we obtain the characteristic polynomial $K_m(x)=$
\be
|K_m-x I_m|=(-1)^{m}(a_m+a_{m-1} x + \ldots +a_1 x^{m-1}+x^m)=(-1)^{m}\sum_{j=0}^m a_{m-j}\,x^j.
\label{eq:chr2}
\ee
Hence if $a_m\neq 0$, then the inverse recurrence matrix $K_m^{-1}$ exists, and is given by
\[
K_m^{-1}=\left(
\begin{array}{ccccc}
 0 & 0 &\ldots& 0 & -\frac{1}{a_m} \\
 1 & 0 & \ldots & 0 & -\frac{a_1}{a_m} \\
 0 & 1 &\ldots& 0 & -\frac{a_2}{a_m} \\
  \vdots & \vdots & \ddots & \vdots & \vdots \\
  0 & 0 & \ldots& 1 & -\frac{a_{m-1}}{a_m} \\
\end{array}
\right),
\]
with characteristic polynomial
\be
K_m^{-1}(x)=\frac{(-1)^{m}}{a_m}\left (1+a_1 x+\ldots +a_{m-1} x^{m-1}+a_m x^m\right )= \frac{(-1)^{m}}{a_m}\sum_{j=0}^m a_j x^j,
\label{eq:chr3}
\ee
and the expressions for $R_m$, $R_m^{-1}$, $R_m(x)$, and $R_m^{-1}(x)$ follow.

If the $m$ eigenvalues of the recurrence matrix $K_m$, denoted by $\lambda_{m\,1}, \lambda_{m\,2}, \ldots, \lambda_{m\,m}$, are non-zero, real, algebraic, distinct, and listed in descending order in terms of absolute value, then the sequences
\[
Y_m^{(j)}=\left (1,\lambda_{m\,j},\lambda_{m\,j}^2,\lambda_{m\,j}^3,\ldots \right ),\qquad j=1,2,\ldots m,
\]
form a basis for the solution space of all possible sequences satisfying the recurrence, for any possible initial values. Similarly the ordered eigenvalues $\lambda_{m\,m}^{-1}, \ldots, \lambda_{m\,1}^{-1}$ of the recurrence matrix $K^{-1}_m$, form  a basis for the solution space of all possible sequences satisfying the inverse recurrence relation. It follows that each sequence term generated by the recurrence relation can be expressed as a linear combination of powers of the eigenvalues of the recurrence matrix. With regard to our sequences $\mathscr{F}_r^{(m,j)}$, and
recurrence matrix $R_m$, we have $\lambda_{m\,i}=1/\mu_{m\,i}$, and hence (\ref{eq:ex1}).

The statements concerning the entries of $R_m^m$ and that $M_{o}(m,0)$ or $M_{e}(m,0)$ are integer matrices when $2m+1$ is a prime number follow from the property
\[
\frac{2m+1}{2k+1}\binom{m+k}{2k}\in \mathbb{N}.
\]
\end{proof}

\begin{thm2}
For $m$ a positive integer, and $0\leq j\leq m-1$, we have
\be
\mathscr{F}_{r}^{(m,m-j)}=\sum_{k=0}^{j}\binom{j+k+1}{2k+1}\mathscr{F}_{r-k}^{(m,m)},\qquad \mathscr{G}_{r}^{(m,m-j)}=-\sum_{k=0}^{j}\binom{j+k}{2k}\mathscr{F}_{r-k}^{(m,m)},
\label{eq:s9}
\ee
so that each of the terms in the sequence $\{\mathscr{F}_r^{(m,j)}\}_{r=1}^{\infty}$ can be expressed as a binomial coefficient linear combination of $(m+1-j)$ terms from the sequence $\{\mathscr{F}_r^{(m,m)}\}_{r=1}^{\infty}$, where we note that $\mathscr{F}_{r}^{(m,m)}=-\mathscr{G}_{r}^{(m,m)}$ $\forall\, r\in\mathbb{Z}$.

The generating functions for $\mathscr{F}_{r}^{(m,m-j)}$ and $\mathscr{G}_{r}^{(m,m-j)}$ are given by
\be
\sum_{r=0}^\infty \mathscr{F}_{r}^{(m,m-j)}x^r= \frac{(2m+1)\mathcal{Q}_{j}(x)}{P_m(x)},
\qquad
\sum_{r=0}^\infty \mathscr{G}_{r}^{(m,m-j)}x^r=-\frac{(2m+1)\mathcal{P}_{j}(x)}{P_m(x)},
\label{eq:s11}
\ee
so that the sum of the numerator coefficients of each generating function is a Fibonacci number.

Inverting the expressions for $\mathscr{F}_r^{(m,j)}$ and $\mathscr{G}_r^{(m,j)}$ in $(\ref{eq:s9})$, we obtain
\be
\mathscr{F}_{r}^{(m,j)}=\sum_{k=0}^{j-1}\frac{2j-1}{2k+1}\binom{j+k-1}{2k}\mathscr{F}_{r-k}^{(m,1)},\quad 1\leq j \leq m,
\label{eq:s12}
\ee
and
\be
\mathscr{G}_{r}^{(m,j)}=\sum_{k=0}^{j}\frac{j}{k}\binom{j+k-1}{2k-1}\mathscr{F}_{r-k}^{(m,1)},\quad 1\leq j \leq m-1.
\label{eq:s13}
\ee
\end{thm2}
\begin{corollary}
In terms of the roots $\mu_{m\,1},\ldots,\mu_{m\,m}$ of the polynomial equation $P_m(x)=0$, we have
\be
\mathscr{F}_r^{(m,m)}=-\mathscr{G}_r^{(m,m)}=-\sum_{k=1}^m \frac{2m+1}{\mu_{m\, k}^{r} \prod_{j\neq k}(\mu_{m\, k}-\mu_{m\, j})}
\label{eq:u7}
\ee
\be
=-(2m+1)\sum_{k=1}^m \left (2\cos\left ( \frac{2\pi k}{2m+1}\right )-2 \right )^{-r} \prod_{j\neq k}\left (2\cos\left ( \frac{2\pi k}{2m+1}\right )-2\cos\left ( \frac{2\pi j}{2m+1}\right )\right )^{-1},
\label{eq:u8}
\ee
\be
\mathscr{F}_r^{(m,1)}=\mu_{m\, 1}^{1-r}+\ldots+\mu_{m\, m}^{1-r}
=\sum_{k=1}^m \left (2\cos\left ( \frac{2\pi k}{2m+1}\right )-2 \right )^{1-r},
\label{eq:t33}
\ee
\[
\mathscr{F}_{r}^{(m,j)}=\sum_{t=1}^{m}\left (\mu_{m\, t} \right )^{1-r}P_{j-1}\left (\mu_{m\, t}  \right )
=\sum_{t=1}^{m}\left (\mu_{m\, t} \right )^{1-r}\prod_{k=1}^{j-1}\left (\mu_{m\, t}-\mu_{(j-1)\, k}  \right )
\]
\be
=\sum_{t=1}^{m}\left (\mu_{m\, t} \right )^{1-r}\prod_{k=1}^{j-1}\left (\phi_{m\, t}-\phi_{(j-1)\, k}\right )
=\sum_{t=1}^{m}\left (\mu_{m\, t} \right )^{1-r}S_{2j-2}\left (2\cos\left ( \frac{\pi t}{2m+1}\right ) \right )
\label{eq:f34}
\ee
\[
=(-1)^{r-1}\sum _{t=1}^m \frac{ 2 \sin \left(\frac{ (2 j-1)\pi t}{2 m+1}\right)}{\left(2 \sin \left(\frac{\pi  t}{2
   m+1}\right)\right)^{2 r-1}}
=\sum_{t=1}^{m}\left (\mu_{m\, t} \right )^{1-r}V_{j-1}\left (\phi_{m\, t}  \right ),
\]
and
\[
\mathscr{G}_{r}^{(m,j)}=\sum_{t=1}^{m}\left (\mu_{m\, t} \right )^{1-r}Q_j\left (\mu_{m\, t}  \right )
=\sum_{t=1}^{m}\left (\mu_{m\, t} \right )^{1-r}\prod_{k=1}^j\left (\mu_{m\, t}-\nu_{j\, k}  \right )
\]
\be
=\sum_{t=1}^{m}\left (\mu_{m\, t} \right )^{1-r}\mathcal{C}_{2j}\left (2\cos\left ( \frac{\pi t}{2m+1}\right ) \right )
=(-1)^{r-1}\sum _{t=1}^m \frac{ 2 \cos \left(\frac{2 j\pi t}{2 m+1}\right)}{\left(2 \sin \left(\frac{\pi  t}{2
   m+1}\right)\right)^{2 r-2}}.
\label{eq:f35}
\ee
\end{corollary}
\begin{proof}[Proof of Theorem 2]
The denominator polynomial of $\frac{1}{2m+1}P_m(x)$ in the generating functions for $\mathscr{F}_r^{(m,m)}$ (and so $\mathscr{G}_r^{(m,m)}$) is a direct consequence of
$\frac{-1}{2m+1}P_m(x)$ being the recurrence polynomial for $\mathscr{F}_r^{(m,m)}$. The simple numerator follows from the starting vector for $\mathscr{F}_r^{(m,m)}$
in the initial value matrix consisting of $(0,0,0,\ldots,-1)$.

The two identities in $(\ref{eq:s9})$ follow directly from (6.8) and (6.9) of Lemma 6.2 in \cite{fleck}. Applying $(\ref{eq:s9})$ to the generating
function polynomial for $\mathscr{F}_r^{(m,m)}$ thus establishes the numerator polynomials in the generating functions of $\mathscr{F}_r^{(m,m-j)}$ and $\mathscr{G}_r^{(m,m-j)}$ in $(\ref{eq:s11})$.

With the binomial matrices of initial conditions $M_o(m,r)$ and $M_e(m,r)$, so defined, it is possible to invert the identities in $(\ref{eq:s9})$ using a binomial convolution to obtain $(\ref{eq:s12})$ and  $(\ref{eq:s13})$.

To see $(\ref{eq:u7})$, we know that $P_m(x)$ factors as $P_m(x)=\prod_{k=1}^m (x-\mu_{m\,k})$.  Our contour integrals from the
generating function have the form
$${1 \over {2\pi i}}\oint {{(2m+1)z} \over {z^{r+1}P_m(z)}}{\rm d}z.$$
The contour encloses the origin and at least the interval $(-4,0]$ of the negative axis,
in order to contain all of the simple poles of $P_m(z)$ and the higher order pole at the
origin.

Assuming the partial fractional decomposition
$${1 \over {P_m(x)}}=\sum_{k=1}^m {c_k \over {x-\mu_{m\,k}}},$$
we now show that $c_k=1/P_m'(\mu_{m\,k})$, where the condition that $P_m(x)$ has distinct roots implies
that $P_m'(\mu_{m\,k}) \neq 0$.  Using the form with lowest common denominator $P_m(x)$, we have
$${1 \over {P_m(x)}}=\sum_{k=1}^m {c_k \over {x-\mu_{m\,k}}}={{\sum_{k=1}^m c_k \prod_{j=1,j\neq k}^m (x-\mu_{m\,j})} \over {P_m(x)}}.$$
Then $1=\sum_{k=1}^m c_k \prod_{j=1,j\neq k}^m (x-\mu_{m\,j})$, and, evaluating at $\mu_{m\,n}$,
$1 \leq n \leq m$, gives $1=c_n \prod_{j=1,j\neq n}^m (\mu_{m\,n}-\mu_{m\,j})=c_n P_m'(\mu_{m\,n})$.
Hence $c_n=1/P_m'(\mu_{m\,n})$.

We have determined that
$${1 \over {P_m(x)}}=\sum_{k=1}^m {c_k \over {x-\mu_{m\,k}}}=\sum_{k=1}^m {1 \over {(x-\mu_{m\,k})}}{1 \over
{P_m'(\mu_{m\,k})}}$$
wherein $P_m'(\mu_{m\,k})=\prod_{j=1,j\neq k}^m (\mu_{m\,k}-\mu_{m\,j})$.
Then
$${1 \over {2\pi i}}\oint {{2m+1} \over {z^{r+1}P_m(z)}}{\rm d}z={1 \over {2\pi i}}\oint {{2m+1} \over {z^{r+1}}}\sum_{k=1}^m {1 \over {(z-\mu_{m\,k})}}{1 \over {P_m'(\mu_{m\,k})}}{\rm d}z,$$
and the residue at $z=\mu_{m\,k}$ is given by
$${{(2m+1)} \over {\mu_{m\,k}^{r+1}}}{1 \over {P_m'(\mu_{m\,k})}}={{(2m+1)} \over {\mu_{m\,k}^{r+1}}}{1 \over
{\prod_{j=1,j\neq k}^m (\mu_{m\,k}-\mu_{m\,j})}}.$$
The pole at the origin gives the $\mathscr{F}_r^{(m,m)}$
term generally.  Using $2\pi i$ times the sum of all residues gives
\be
\mathscr{F}_r^{(m,m)}+\sum_{k=1}^m \frac{2m+1}{\mu_{m\,k}^{r} \prod_{j\neq k}(\mu_{m\,k}-\mu_{m\,j})}=0,\qquad r=1,2,3, \ldots,
\label{eq:main}
\ee
and hence the result and $(\ref{eq:u8})$.

The identity $\mathscr{F}_r^{(m,1)}=\mu_{m\, 1}^{1-r}+\mu_{m\, 2}^{1-r}+\ldots +\mu_{m\, m}^{1-r}$ in $(\ref{eq:t33})$ follows similarly to $(\ref{eq:u7})$. Combining these results with $(\ref{eq:s12})$ and $(\ref{eq:s13})$ then gives
\[
\mathscr{F}_r^{(m,j)}=\sum_{t=1}^{m}\sum_{k=0}^{j-1}\frac{2j-1}{2k+1}\binom{j+k-1}{2k}\left (2\cos\left ( \frac{2\pi t}{2m+1}\right )-2 \right )^{k+1-r},\quad 1\leq j \leq m.
\]
\[
=\sum_{t=1}^{m}\left (2\cos\left ( \frac{2\pi t}{2m+1}\right )-2 \right )^{1-r}\sum_{k=0}^{j-1}\frac{2j-1}{2k+1}\binom{j+k-1}{2k}\left (2\cos\left ( \frac{2\pi t}{2m+1}\right )-2 \right )^{k}
\]
\[
=\sum_{t=1}^{m}\left (2\cos\left ( \frac{2\pi t}{2m+1}\right )-2 \right )^{1-r}P_{j-1}\left (2\cos\left ( \frac{2\pi t}{2m+1}\right )-2 \right )
\]
\[
=\sum_{t=1}^{m}\left (\mu_{m t} \right )^{1-r}P_{j-1}\left (\mu_{m t}  \right )
=\sum_{t=1}^{m}\left (\mu_{m t} \right )^{1-r}\prod_{k=1}^{j-1}\left (\mu_{m t}-\mu_{(j-1) k}  \right ),
\]
and
\[
\mathscr{G}_{r}^{(m,j)}=\sum_{k=0}^{j}\frac{j}{k}\binom{j+k-1}{2k-1}\mathscr{F}_{r-k}^{(m,1)},\quad 1\leq j \leq m-1,
\]
\[
=\sum_{t=1}^{m}\sum_{k=0}^{j}\frac{2j}{2k}\binom{j+k-1}{2k-1}\left (2\cos\left ( \frac{2\pi t}{2m+1}\right )-2 \right )^{k+1-r}
\]
\[
=\sum_{t=1}^{m}\left (2\cos\left ( \frac{2\pi t}{2m+1}\right )-2 \right )^{1-r}
\sum_{k=0}^{j}\frac{2j}{2k}\binom{j+k-1}{2k-1}\left (2\cos\left ( \frac{2\pi t}{2m+1}\right )-2 \right )^{k}
\]
\[
=\sum_{t=1}^{m}\left (2\cos\left ( \frac{2\pi t}{2m+1}\right )-2 \right )^{1-r}Q_j\left (2\cos\left ( \frac{2\pi t}{2m+1}\right )-2 \right )
\]
\[
=\sum_{t=1}^{m}\left (\mu_{m t} \right )^{1-r}Q_j\left (\mu_{m t}  \right )
=\sum_{t=1}^{m}\left (\mu_{m t} \right )^{1-r}\prod_{k=1}^j\left (\mu_{m t}-\nu_{j k}  \right ).
\]
We have thus established $(\ref{eq:f34})$ and $(\ref{eq:f35})$.

\end{proof}

\begin{rmk}[to Theorem 2]
For $1\leq j\leq m$, the sequence terms $\mathscr{F}_r^{(m,j)}$ and $\mathscr{G}_r^{(m,j)}$ naturally occur in matrix powers of particular circulant matrices as described in \cite{fleck}. In this setting there exists an additional sequence $\mathscr{G}_r^{(m,0)}$, which appears in the leading diagonal of the powers of this circulant matrix, thus bringing the total number of sequences to $n=2m+1$. This sequence satisfies the relation
\be
\mathscr{G}_r^{(m,0)}=2\mathscr{F}_r^{(m,1)}=-2\sum_{k=1}^m \mathscr{G}_r^{(m,k)}.
\label{eq:t26}
\ee
so it also obeys the recurrence relation $R_m$, and we can write
\be
2\mathscr{G}_{r}^{(m,j)}=-\sum_{k=0}^{j}\frac{j}{k}\binom{j+k-1}{2k-1} \mathscr{G}_{r-k}^{(m,0)},\quad 1\leq j \leq m-1.
\label{eq:t27}
\ee
By Lemma 6.1 of \cite{fleck}, (see also Lemma 3.3 of \cite{trio}) these type of recurrence relations can be written as half-weighted determinants. For example $\mathscr{F}_r^{(m,1)}$ has the form
\[
\mathscr{F}_{r}^{(m,1)}=(-1)^{r+m-1}\left |
\begin{array}{ccccccc}
\binom{m}{1} & 1 & 0 & 0 & 0 & \ldots & 0\\
\binom{m+1}{3} & h_1 & 1 & 0 & 0 &  \ldots & 0 \\
\binom{m+2}{5} & h_2 & h_1 & 1 & 0 &  \ldots & 0 \\
\vdots & \vdots & \vdots & \vdots & \vdots &  \ddots & \vdots\\
\binom{m+r-1}{2r-1} & h_{r-1} & h_{r-2} & h_{r-3} & h_{r-4} & \ldots & 1\\
\binom{m+r}{2r+1} & h_r & h_{r-1} & h_{r-2} & h_{r-3} & \ldots & h_1\\
\end{array}
\right |.
\]
Hence for $r\in\mathbb{N}$, the sequence terms $\mathscr{F}_r^{(m,j)}$, $1\leq j\leq,m$, are rational numbers corresponding to a rational multiple of either the $r$ or the $(r+1)$-dimensional volume of the simplex with vertex coordinates given by the row entries of the determinant.
\end{rmk}
\begin{lemma}[Ratio lemma]
For $j\geq 0$, and $1\leq t \leq m$, we have
\be
\mu_{m\, t}\,P_j(\mu_{m \,t})=\mu_{m \,(j+1)t}-\mu_{m\, j t},\qquad \mu_{m\, t}\,Q_j(\mu_{m \,t})=\mu_{m \,(j+1)t}-2\mu_{m\, j t}+\mu_{m \,(j-1)t},
\label{eq:l1}
\ee
and
\be
 \frac{\mu_{m\, (m-2j+1)t}-\mu_{m\, (m-2j)t}}{\mu_{m\, (m-j+1)t}-\mu_{m\, (m-j)t}}
 =\frac{\phi_{m\, (m-2j+1)t}-\phi_{m\, (m-2j)t}}{\phi_{m\, (m-j+1)t}-\phi_{m\, (m-j)t}}=\phi_{m\, j t},
\quad 1\leq j \leq \left [{\scriptstyle\frac{m}{2}}\right ],
\label{eq:l2}
\ee
\be
\frac{\mu_{m\, (2j -m)t}-\mu_{m\, (2j -m-1)t}}{\mu_{m\, (m-j +1)t}-\mu_{m\, (m-j)t}}=
\frac{\phi_{m\, (2j-m)t}-\phi_{m\, (2j-m-1)t}}{\phi_{m\, (m-j+1)t}-\phi_{m\, (m-j)t}}=-\phi_{m\, j t},\quad
\left [{\scriptstyle\frac{m}{2}}\right ]+1\leq j \leq m.
\label{eq:l3}
\ee
Moreover we have
\be
\mu_{m\, t}\prod_{k=1}^{j-1}\left (\mu_{m\, t}-\mu_{(j-1)\, k}\right )=\mu_{m\, j t}-\mu_{m\,(j-1)t}=\phi_{m\, j t}-\phi_{m\,(j-1)t},\qquad
\label{eq:l38}
\ee
\[
\mu_{m\, t} \prod_{k=1}^j\left (\mu_{m\, t}-\nu_{j\, k}  \right )=\phi_{m\, (j+1) t}-2\phi_{m\, j t}+\phi_{m\,(j-1)t},
\]
so that
\be
\mathscr{F}_{r}^{(m,j)}
=\sum_{t=1}^{m}\left (\mu_{m\, t} \right )^{-r}\left (\phi_{m\, j t}-\phi_{m\,(j-1)t}\right ),
\label{eq:l36}
\ee
\be
\mathscr{G}_{r}^{(m,j)}
=\sum_{t=1}^{m}\left (\mu_{m\, t} \right )^{-r}\left (\phi_{m\, (j+1) t}-2\phi_{m\, j t}+\phi_{m\,(j-1)t}\right ).
\label{eq:l37}
\ee

\end{lemma}
\begin{defn}[of convergent vector sequences in $\mathbb{Q}^m$]
We define the convergent vector sequences ${\bf \Psi}_r^{(m)}$, and their limit points ${\bf \Phi}^{(m)}$, such that
${\bf \Phi}^{(m)}=\left ( \phi_{m \,1},\phi_{m \,2},\ldots,\phi_{m\, m}\right ),$ and ${\bf \Psi}_r^{(m)}=$
\[
\left ( \frac{\mathscr{F}^{(m,m-1)}_{r}}{\mathscr{F}^{(m,m)}_r},\ldots, \frac{\mathscr{F}^{(m,m-2r+1)}_{r}}{\mathscr{F}^{(m,m-r+1)}_{r}} ,\ldots \frac{-\mathscr{F}^{(m,1)}_r}{\mathscr{F}^{(m,[m/2]+1)}_{r}}, \ldots,\frac{-\mathscr{F}^{(m,2r-m)}_{r}}{\mathscr{F}^{(m,m-r+1)}_{r}},\ldots,\frac{-\mathscr{F}^{(m,m)}_r}{\mathscr{F}^{(m,1)}_r} \right ),
\]
where from the definition of $\mathscr{N}_r^{(m,j)}$, we can replace $\mathscr{F}_r^{(m,j)}$ by $\mathscr{N}_r^{(m,j)}$ in the above expression
for ${\bf \Psi}_r^{(m)}$, without affecting the ratios.

\end{defn}
\begin{thm3}[Limit theorem]
We have $\mathop{\lim}_{r\rightarrow \infty}{\bf \Psi}_r^{(m)}=$
\[
\left ( \frac{\mu_{m\, (m-1)}-\mu_{m\, (m-2)}}{\mu_{m\, m}-\mu_{m\,(m-1)}}, \ldots, \frac{\mu_{m\, (m-2r+1)}-\mu_{m \,(m-2r)}}{\mu_{m\, (m-r+1)}-\mu_{m \,(m-r)}} ,\ldots\frac{-(\mu_{m \,1}-\mu_{m\, 0})}{\mu_{m \,([m/2]+1)}-\mu_{m \,[m/2]}}, \ldots\right.
\]
\be
\left.\ldots,\frac{-\left(\mu_{m\, (2r-m)}-\mu_{m\, (2r-m-1)}\right)}{\mu_{m\, (m-r+1)}-\mu_{m\, (m-r)}},\ldots,
\frac{-\left(\mu_{m\, m}-\mu_{m \, (m-1)}\right)}{\mu_{m\, 1}-\mu_{m\, 0}} \right )={\bf \Phi}^{(m)},
\label{eq:t32}
\ee
so that
\[
\mathop{ \lim}_{r\rightarrow \infty}\left (x^2-\frac{\mathscr{F}^{(m,m-1)}_{r}}{\mathscr{F}^{(m,m)}_r}\,x+1\right )\times \ldots\times \left (x^2+\frac{\mathscr{F}^{(m,m)}_{r}}{\mathscr{F}^{(m,1)}_r}\,x+1\right )
\]
\[
=\mathop{\lim}_{r\rightarrow \infty}\left (x^2-\frac{\mathscr{N}^{(m,m-1)}_{r}}{\mathscr{N}^{(m,m)}_r}\,x+1\right )\times \ldots\times \left (x^2+\frac{\mathscr{N}^{(m,m)}_{r}}{\mathscr{N}^{(m,1)}_r}\,x+1\right )
=x^{2m}+\ldots + x+1,
\]
where we note that as $(\mu_{m\,j}-\mu_{m\,k})=(\phi_{m\,j}-2-(\phi_{m\,k}-2))=(\phi_{m\,j}-\phi_{m\,k})$, we can replace $\mu_{m\,i}$ by $\phi_{m\,i}$
in the above expression for ${\bf \Psi}_r^{(m)}$.
\end{thm3}
\begin{corollary1}[Convergence corollary]
For $1\leq j \leq m$, let $\sigma_{m\,j}=\mu_{m\,1}/\mu_{m\,j}$, so that $\sigma_{m\,1}=1$, $\sigma_{m\,2}=(\phi_{m\,1}+1)^{-1}$, and for $j\geq 2$, $0<|\sigma_{m\,j}|<1$. Let $a_t^{(m,j)}=\phi_{m\, j t}-\phi_{m\,(j-1)t}$, so that $a_t^{(m,j)}$ is the coefficient of $\mu_{m\,t}^{-r}$ in the closed form expression for $\mathscr{F}_r^{(m,j)}$ given in (\ref{eq:f34}) and (\ref{eq:l36}). Now define
\[
B_{j\,k}=\frac{4m\left (|\phi_{m\, j }-\phi_{m\,(j-1)}|+|\phi_{m\, k }-\phi_{m\,(k-1)}|\right )
-2|\phi_{m\, j }-\phi_{m\,(j-1)}||\phi_{m\, k }-\phi_{m\,(k-1)}|}{|\phi_{m\, k }-\phi_{m\,(k-1)}|^2},
\]
and let $r_k$ be the least positive integer satisfying
\[
\left |\frac{1}{\phi_{m\,1}+2}\right |^{r_k}\leq \frac{|\phi_{m\, k }-\phi_{m\,(k-1)}|}{2\left (4m-|\phi_{m\, k }-\phi_{m\,(k-1)}|\right )}.
\]
Then for $r> r_k$ we have
\[
\left |\frac{\mathscr{F}^{(m,j)}_{r}}{\mathscr{F}^{(m,k)}_r}-\frac{a_1^{(m,j)}}{a_1^{(m,k)}}\right |\leq 2B_{j\,k}\left |\sigma_{m\,2}\right |^{r},
\]
so that $\mathscr{F}^{(m,j)}_{r}/\mathscr{F}^{(m,k)}_{r}$, approximates $a_1^{(m,j)}/a_1^{(m,k)}$, with remainder $\leq 2B_{j\,k}\left |\sigma_{m\,2}\right |^{r}$. Furthermore, let $r^{\prime}$ be the maximum of all the $r_k$, $B^\prime$ be the maximum of all the $B_{j\,k}$,
and for $\epsilon>0$, let $r^*$ be the least positive integer with $r^*> r^{\prime}$ satisfying
\[
2B^{\prime}\sqrt{m}\left |\sigma_{m\,2}\right |^{r^*}<\epsilon.
 \]
 Then $|{\bf \Psi}_r^{(m)}-{\bf \Phi}^{(m)}|<\epsilon$ for $r>r^*$, and (by the standard definition in Section 1) the vector sequence ${\bf \Psi}_r^{(m)}$ is a sequence of vector convergents to the limit point ${\bf \Phi}^{(m)}$.

When $j=m-2u+1$ with $k=m-u+1=j+u$, as described in (\ref{eq:l2}), or similarly as in (\ref{eq:l3}), then $|a_1^{(m,j)}/a_1^{(m,k)}|=|\phi_{m\,u}|$,
and the above simplifies to
\[
B_{j\,k}=\frac{4m\left (|\phi_{m\,u}|+1\right )}{|\phi_{m\,k}-\phi_{m\,(k-1)}|}-2|\phi_{m\,u}|,
\]
so that for $r> r_k$ we have
\[
\left |\frac{\mathscr{F}^{(m,j)}_{r}}{\mathscr{F}^{(m,k)}_r}-\frac{a_1^{(m,j)}}{a_1^{(m,k)}}\right |
\leq \left (\frac{8m\left (|\phi_{m\,u}|+1\right )}{|\phi_{m\,k}-\phi_{m\,(k-1)}|}-2|\phi_{m\,u}|\right )
\left |\frac{1}{\phi_{m\,1}+2}\right |^{r}.
\]
\end{corollary1}
\begin{corollary2}[Continued fraction corollary]
For each sequence of convergents to the limit point $a_1^{(m,j)}/a_1^{(m,k)}$ generated by $\mathscr{F}^{(m,j)}_{r}/\mathscr{F}^{(m,k)}_r$, there will exist a corresponding sequence of simple continued fractions, which will converge to the simple continued fraction expansion of the limit point itself.
\end{corollary2}

\begin{proof}[Proof of Lemma 3.2]
From $(\ref{eq:pf11})$ and $(\ref{eq:pf12})$ in the proof of Theorem 1, we can write
\[
\mu_{m\, t}\,P_j(\mu_{m\, t})=\mu_{m\, t}\left (U_j\left ( \cos{\left (\frac{2\pi t}{2m+1}\right )}\right )
+U_{j-1}\left ( \cos{\left (\frac{2\pi t}{2m+1}\right )}\right )\right )
\]
and using the identity $\mu_{m\, t}=-4\sin^2{\left (\frac{\pi t}{2m+1}\right )},$
yields
\[
\mu_{m\, t}\,P_j(\mu_{m\, t})=-4\sin^2{\left (\frac{\pi t}{2m+1}\right )}
\left (\frac{\sin{\left (2\pi (j+1)t/(2m+1)\right ) }}{\sin{\left (2\pi t/(2m+1)\right )}}
+\frac{\sin{\left (2\pi j t/(2m+1)\right ) }}{\sin{\left (2\pi t/(2m+1) \right )}}\right )
\]
\[
=\frac{-2\sin{\left (\pi t/(2m+1) \right )}}{\cos{\left (\pi t/(2m+1) \right )}}
\left( \sin{\left (\frac{2\pi t(j+1)}{2m+1}\right ) } +\sin{\left (\frac{2\pi t j}{2m+1}\right ) } \right )
\]
\[
=\left (\cos{\left(\frac{\pi t}{2 m+1}\right)} \right )^{-1}\left (\cos{\left(\frac{\pi  t(2 j+3)}{2 m+1}\right)}
-\cos{\left(\frac{\pi  t(2 j-1)}{2 m+1}\right)}\right ),
\]
\be
=\left (\cos{\left(\frac{\pi t}{2 m+1}\right)} \right )^{-1}\left (T_{2j+3}\left(\cos{\left (\frac{\pi t}{2 m+1}\right )}\right)
-T_{2j-1}\left(\cos{\left (\frac{\pi t }{2 m+1}\right )}\right)\right ),
\label{eq:pfl2}
\ee
by the definition of $T_n(x)$. Applying the relation $T_{n+1}(x)=2x T_n(x)-T_{n-1}(x)$ from $(\ref{eq:L5})$ of Lemma 1.1 and cancelling, we obtain
\[
T_{2j+3}\left(\cos{\left (\frac{\pi t}{2 m+1}\right )}\right)
-T_{2j-1}\left(\cos{\left (\frac{\pi t }{2 m+1}\right )}\right)
\]
\be
=2\cos{\left(\frac{\pi t}{2 m+1}\right)} \left (T_{2j+2}\left(\cos{\left (\frac{\pi t}{2 m+1}\right )}\right)
-T_{2j}\left(\cos{\left (\frac{\pi t}{2 m+1}\right )}\right)  \right ),
\label{eq:pfl1}
\ee
and substituting into $(\ref{eq:pfl2})$ then gives $\mu_{m \,t}\,P_j(\mu_{m\, t})=$
\[
T_{2j+2}\left(\cos{\left (\frac{\pi t}{2 m+1}\right )}\right)-T_{2j}\left(\cos{\left (\frac{\pi t}{2 m+1}\right )}\right)
=\mu_{m\, (j+1)t}-\mu_{m\, j t}=\phi_{m\, (j+1)t}-\phi_{m\, j t},
\]
which is $(\ref{eq:l1} )$.

For $(\ref{eq:l2} )$, we write
\[
\phi_{m \,r t}\left (\phi_{m \,(m-r+1)t}-\phi_{m\, (m-r)t}\right )=T_{r}(\phi_{m\, t}/2)\left ( T_{m-r+1}(\phi_{m\, t}/2)- T_{m-r}(\phi_{m\, t}/2) \right ),
\]
and using the identity $2T_{m}(x)T_n(x)=T_{m+n}(x)+T_{|m-n|}(x)$ from $(\ref{eq:L5})$ of Lemma~1.1,
after cancellation then gives us
\[
2\left ( T_{m-2r+1}\left(\phi_{m\, t}/2\right )-T_{m-2r}\left(\phi_{m\, t}/2\right)+T_{m+1}\left(\phi_{m\, t}/2\right)-T_{m}\left(\phi_{m\, t}/2\right)\right )
\]
\[
=\phi_{m\, (m-2r+1)t}-\phi_{m (m-2r)t}+\phi_{m\, (m+1)t}-\phi_{m\, m t}=\phi_{m\, (m-2r+1)t}-\phi_{m\, (m-2r)t},
\]
as required. Similarly we deduce $(\ref{eq:l3} )$.

\end{proof}
\begin{proof}[Proof of Theorem 3]
From $(\ref{eq:f34})$ of the Corollary to Theorem 2, we have
\[
\mathscr{F}_{r}^{(m,j+1)}=\sum_{t=1}^{m}\left (\mu_{m\, t} \right )^{1-r}P_j\left (\mu_{m\, t}  \right )
=\sum_{t=1}^{m}\left (\mu_{m\, t} \right )^{1-r}\prod_{k=1}^j\left (\mu_{m\, t}-\mu_{j\, k}  \right ),
\]
and we consider
\[
\frac{\mathscr{F}^{(m,m-2k+1)}_{r}}{\mathscr{F}^{(m,m-k+1)}_{r}}
=\frac{\mathscr{N}^{(m,m-2k+1)}_{r}}{\mathscr{N}^{(m,m-2k+1)}_{r}}
=\frac{\sum_{t=1}^{m}\left (\mu_{m\, t} \right )^{1-r}P_{m-2k}\left (\mu_{m\, t}  \right ) }{\sum_{t=1}^{m}\left (\mu_{m\, t} \right )^{1-r}P_{m-k}\left (\mu_{m\, t}  \right ) },
\]
by $(\ref{eq:l1})$ of Lemma 3.2.

In the numerator sum above, for large positive values of $r$, the $\left (\mu_{m\, 1} \right )^{1-r}$ factor will dominate
as the $\mu_{m\, t}$ are ordered in terms of increasing absolute value. Hence, as $r\rightarrow \infty$ the above expression will converge to the ratio of the coefficients of $\left (\mu_{m\, 1} \right )^{1-r}$ in the numerator and denominator. Therefore we can write
\[
\lim_{r\rightarrow \infty}\left (\frac{\mathscr{F}^{(m,m-2k+1)}_{r}}{\mathscr{F}^{(m,m-k+1)}_{r}}\right )
=\lim_{r\rightarrow \infty}\left (\frac{\mathscr{N}^{(m,m-2k+1)}_{r}}{\mathscr{N}^{(m,m-k+1)}_{r}}\right )
=\lim_{r\rightarrow \infty}\left (\frac{\left (\mu_{m\,1} \right )^{1-r}P_{m-2k}\left (\mu_{m\,1}  \right ) }{\left (\mu_{m\,1} \right )^{1-r}P_{m-k}\left (\mu_{m\,1}  \right ) }\right )
\]
\[
=\frac{\left (\mu_{m\,1} \right )P_{m-2k}\left (\mu_{m\,1}  \right ) }{\left (\mu_{m\,1} \right )P_{m-k}\left (\mu_{m\,1}\right )}
 =\frac{\phi_{m\, (m-2k+1)}-\phi_{m\, (m-2k)}}{\phi_{m\, (m-k+1)}-\phi_{m\, (m-k)}}=\phi_{m\, k},
\]
by $(\ref{eq:l1})$ and $(\ref{eq:l2})$ of Lemma 3.2, and hence the result. The polynomial identity in the limit then follows immediately .

Regarding the first Corollary, let $\mathscr{H}^{(m,j)}_{r}=\mathscr{F}^{(m,j)}_{r}\mu_{m\,1}^{r}$, so that by (\ref{eq:l36})
\[
\mathscr{H}^{(m,j)}_{r}=a_1^{(m,j)}+a_2^{(m,j)}\sigma_{m\,2}^{r}+\ldots +a_m^{(m,j)}\sigma_{m\,m}^{r}
\]
and
\[
\left |\frac{\mathscr{F}^{(m,j)}_{r}}{\mathscr{F}^{(m,k)}_r}-\frac{a_1^{(m,j)}}{a_1^{(m,k)}}\right |
=\left |\frac{\mathscr{H}^{(m,j)}_{r}}{\mathscr{H}^{(m,k)}_r}-\frac{a_1^{(m,j)}}{a_1^{(m,k)}}\right |,
\]
\[
=\left | \frac{\left (\mathscr{H}^{(m,j)}_{r}-a_1^{(m,j)}\right )a_1^{(m,k)}-\left (\mathscr{H}^{(m,k)}_{r}-a_1^{(m,k)}\right )a_1^{(m,j)}}{a_1^{(m,k)}\left (a_1^{(m,k)}+\left (\mathscr{H}^{(m,k)}_{r}-a_1^{(m,k)}\right )\right )}\right |.
\]
Now for $1\leq j\leq m$, we have $\sum_{t=2}^m |a_t^{(m,j)}|\leq 4m-|a_1^{(m,j)}|$, and
\[
\left |\mathscr{H}^{(m,j)}_{r}-a_1^{(m,j)}\right |< \left (|a_2^{(m,j)}|+\ldots |a_m^{(m,j)}|\right )|\sigma_{m\,2}^{r}|,
\]
so that
\[
\left |\frac{\mathscr{H}^{(m,j)}_{r}}{\mathscr{H}^{(m,k)}_r}-\frac{a_1^{(m,j)}}{a_1^{(m,k)}}\right |
\]
\[
<\left |\frac{|a_1^{(m,k)}|(|a_2^{(m,j)}|+\ldots |a_m^{(m,j)}|)+|a_1^{(m,j)}|(|a_2^{(m,k)}|+\ldots |a_m^{(m,k)}|)}
{|a_1^{(m,k)}|\left ( |a_1^{(m,k)}|-\left (  |a_2^{(m,k)}|+\ldots |a_m^{(m,k)}|   \right ) \right )|\sigma_{m\,2}^{r-1}|}\right ||\sigma_{m\,2}^{r-1}|.
\]
Hence for $r>r_k$ with $r_k$ as defined, the results follow.

Regarding the vector distance $|{\bf \Psi}_r^{(m)}-{\bf \Phi}^{(m)}|$, by the definition of $B^\prime$ and $r^\prime$ it follows that for any specific values for $j$ and $k$ with $r>r^\prime$, the remainder term will be $<B^\prime|\sigma_{m\,2}|^{r-1}$. Taking the square root of $m$ times the square of this bound to obtain the Euclidean distance, we see that the remainder term is now $<B^\prime\sqrt{m} |\sigma_{m\,2}|^{r-1}$. As
$|\sigma_{m\,2}|<1$, we can always find such an $r^*>r^\prime$ satisfying the inequality, as required.

To see the second Corollary, for an irrational number $\alpha$, the simple continued fraction algorithm begins with the largest integer not exceeding $\alpha$ and then proceeds to generate what can be thought of as one infinite branch of number plus fraction. For the $(m+1)$ term recurrence relation that generates the sequence of numerators and denominators in $\mathscr{F}^{(m,j)}_{r}/\mathscr{F}^{(m,k)}_{r}$, at each step the continued fraction branch will have number plus $m-1$ fractions, so at each step the number of branches will increase by a factor of $m-1$. For $m\geq 3$, it is expected that in general the resulting sequence of rational numbers will converge at a slower rate than that for the simple continued fraction expansion of $a_1^{(m,j)}/a_1^{(m,k)}$, as guaranteed by Dirichlets theorem (\ref{eq:dirichlet}). However, the resulting sequence can itself always be translated into a sequence of simple continued fraction expansions, and by the theory, this must converge to the unique simple continued fraction expansion of the limit point
$a_1^{(m,j)}/a_1^{(m,k)}$, as required.
\end{proof}

\begin{defn}[of unlaced Fibonacci sequences]
For $1\leq j\leq m$, denote by $\phantom{}^* \mathscr{F}_{r}^{(m,j)}$, the $r$th term in the $j$th unlaced Fibonacci sequence of dimension $m$, defined in terms of the interlaced Fibonacci sequences of dimension $m$ such that
\[
\phantom{}^* \mathscr{F}_{m q+2r-1}^{(m,j)}=\mathscr{F}_{m q+j}^{(m,r)},\qquad 1\leq r \leq \left [\frac{m+1}{2}\right ],
\]
and
\[
\phantom{}^* \mathscr{F}_{m q+2r}^{(m,j)}=\mathscr{F}_{m q+j}^{(m,m+1-r)},\qquad 1\leq r \leq \left [\frac{m}{2}\right ],
\]
so that the unlaced sequences of Fibonacci numbers correspond to the ordered union of the ordered sets
\[
\left \{\phantom{}^* \mathscr{F}_{r}^{(m,j)}\right \}_{r=1}^\infty
=\mathop{\bigcup_{q=0}^{\infty}}_{n = q m+j}\left \{ \mathscr{F}_{n}^{(m,1)},\mathscr{F}_{n}^{(m,m-1)},\mathscr{F}_{n}^{(m,2)},\mathscr{F}_{n}^{(m,m-2)},\ldots \mathscr{F}_{n}^{(m,[m/2]+1)}\right \}.
\]
Here the ordering mimics that displayed in the construction of the convergent vectors terms ${\bf \Psi}_r^{(m)}$, where as before, we define the unlaced numerator terms $\phantom{}^* \mathscr{N}_{r}^{(m,j)}$, to be the non-reduced numerators of the terms $\phantom{}^* \mathscr{F}_{r}^{(m,j)}$.
\end{defn}

\begin{lemma}
For $1\leq a,b\leq m$, $a\neq b$, and with the above construction, the ratios of the terms of the $j$th unlaced Fibonacci sequences of dimension $m$, given by
\[
\left \{\frac{\phantom{}^* \mathscr{F}_{mt+a}^{(m,j)}}{\phantom{}^* \mathscr{F}_{mt+b}^{(m,j)}}\right \}_{t=0}^\infty
=\left \{\frac{\phantom{}^* \mathscr{N}_{mt+a}^{(m,j)}}{\phantom{}^* \mathscr{N}_{mt+b}^{(m,j)}}\right \}_{t=0}^\infty,
\]
are a rational sequence of convergents to some real limit points, including those of the form $2\cos{\left (\frac{2 k Pi}{2m+1}\right )}$.
\end{lemma}
\begin{proof}
The proof follows immediately from the construction of the unlaced Fibonacci sequences of dimension $m$ from the interlaced sequences of dimension $m$.
\end{proof}

\begin{example}[Three-dimensional interlacing case] By construction, the interlacing Fibonacci sequences of dimension 3 are given by
\be
\left(
\begin{array}{ccc}
 \mathscr{F}_3^{(3,1)} & \mathscr{F}_2^{(3,1)} & \mathscr{F}_1^{(3,1)} \\
 \mathscr{F}_3^{(3,2)} & \mathscr{F}_2^{(3,2)} & \mathscr{F}_1^{(3,2)} \\
 \mathscr{F}_3^{(3,3)} & \mathscr{F}_2^{(3,3)} & \mathscr{F}_1^{(3,3)} \\
\end{array}
\right)=
\left(
\begin{array}{ccc}
 -\binom{0}{0} & -3\binom{1}{0} & -5 \binom{2}{0} \\
 -\frac{1}{3}\binom{1}{2} & -\frac{3}{3}\binom{2}{2} & -\frac{5}{3} \binom{3}{2} \\
 -\frac{1}{5} \binom{2}{4} & -\frac{3}{5} \binom{3}{4} & -\frac{5}{5}\binom{4}{4} \\
\end{array}
\right)^{-1}
=
\left(
\begin{array}{ccc}
 2 & -2 & 3 \\
 4 & -3 & 2 \\
 3 & -2 & 1 \\
\end{array}
\right ),
\label{eq:iv}
\ee
and thereafter for $j\in\{1,2,3\}$ by
\[
\mathscr{F}_{r+3}^{(3,j)}=-2\mathscr{F}_{r+2}^{(3,j)}-\mathscr{F}_{r+1}^{(3,j)}-\frac{1}{7}\mathscr{F}_r^{(3,j)},
\]
so that
\[
\left(
\begin{array}{ccc}
 \mathscr{F}_{r+3}^{(3,1)} & \mathscr{F}_{r+2}^{(3,1)} & \mathscr{F}_{r+1}^{(3,1)} \\
 \mathscr{F}_{r+3}^{(3,2)} & \mathscr{F}_{r+2}^{(3,2)} & \mathscr{F}_{r+1}^{(3,2)} \\
 \mathscr{F}_{r+3}^{(3,3)} & \mathscr{F}_{r+2}^{(3,3)} & \mathscr{F}_{r+1}^{(3,3)} \\
\end{array}
\right)
=
\left(
\begin{array}{ccc}
 \mathscr{F}_{r+2}^{(3,1)} & \mathscr{F}_{r+1}^{(3,1)} & \mathscr{F}_r^{(3,1)} \\
 \mathscr{F}_{r+2}^{(3,2)} & \mathscr{F}_{r+1}^{(3,2)} & \mathscr{F}_r^{(3,2)} \\
 \mathscr{F}_{r+2}^{(3,3)} & \mathscr{F}_{r+1}^{(3,3)} & \mathscr{F}_r^{(3,3)} \\
\end{array}
\right)
\left(
\begin{array}{ccc}
 -2 & 1 & 0 \\
 -1 & 0 & 1 \\
 -\frac{1}{7} & 0 & 0 \\
\end{array}
\right).
\]
By (\ref{eq:s11}) of Theorem 2, the generating functions for the sequences $\mathscr{F}_r^{(1)}$,  $\mathscr{F}_r^{(2)}$ and $\mathscr{F}_r^{(3)}$ are given by
\[
\mathscr{F}_r^{(3,1)}:\,\,\,\frac{7 (x^2+4x+3)}{x^3+7 x^2+14x+7}=3-2 x+2 x^2-\frac{17 x^3}{7}+\frac{22 x^4}{7}-\frac{29 x^5}{7}+\frac{269 x^6}{49}-\frac{357 x^7}{49}+\ldots,
\]
\[
\mathscr{F}_r^{(3,2)}:\,\,\,\frac{7 (x+2)}{x^3+7 x^2+14x+7}=2-3 x+4 x^2-\frac{37 x^3}{7}+\frac{49 x^4}{7}-\frac{65 x^5}{7}+\frac{604 x^6}{49}-\frac{802 x^7}{49}+\ldots,
\]
\[
\mathscr{F}_r^{(3,3)}:=\frac{7}{x^3+7 x^2+14 x+7}=1-2 x+3 x^2-\frac{29 x^3}{7}+\frac{39 x^4}{7}-\frac{52 x^5}{7}+\frac{484 x^6}{49}-\frac{643 x^7}{49}+\ldots.
\]
and by (\ref{eq:s12}) $\mathscr{F}_r^{(3,1)}$ and $\mathscr{F}_r^{(3,2)}$ can be expressed in terms of $\mathscr{F}_r^{(3,3)}$ as
\[
 \mathscr{F}_r^{(3,1)}=3 \mathscr{F}_{r}^{(3,3)}+4\mathscr{F}_{r-1}^{(3,3)}+\mathscr{F}_{r-2}^{(3,3)},\qquad \mathscr{F}_r^{(2)}=2 \mathscr{F}_{r}^{(3,3)}+\mathscr{F}_{r-1}^{(3,3)}.
\]
The first few terms of the three unlaced sequences $\phantom{}^* \mathscr{F}_r^{(m,j)}$, with $j\in\{1,2,3,\}$, and their corresponding non-reduced numerator integer sequences $\phantom{}^*\mathscr{N}_r^{(3,j)}$ are given below.

\[
\begin{array}{|c||c||c|c|c|c|c|c|c|c|c|c|}
\hline
{\bf m} & {\bf j/r} & {\bf 1}& {\bf 2} & {\bf 3} & {\bf 4} & {\bf 5} & {\bf 6} &{\bf 7} & {\bf 8} & {\bf 9} & {\bf 10}\\
 \hline\hline
\hbox{}^* \mathscr{F}_r^{(3,1)} & {\bf 1} & 3 & 1 & 2 & \frac{-17}{7} & \frac{-29}{7} & \frac{-37}{7} & \frac{269}{49} & \frac{484}{49} & \frac{604}{49} & \frac{-4406}{343}\\
\hbox{}^* \mathscr{F}_r^{(3,2)} & {\bf 2} & -2 & -2 & -3 & \frac{22}{7} & \frac{39}{7} & \frac{49}{7}  & \frac{-357}{49} & \frac{-643}{49} & \frac{-802}{49}& \frac{5851}{343}  \\
\hbox{}^* \mathscr{F}_r^{(3,3)} & {\bf 3} & 2 & 3 & 4 & \frac{-29}{7} & \frac{-52}{7} & \frac{-65}{7} & \frac{474}{49} & \frac{854}{49} & \frac{1065}{49}  & \frac{-7770}{343}   \\
 \hline
\phantom{}^*\mathscr{N}_r^{(3,1)} & {\bf 1} & 3 & 1 & 2 & -17 & -29 & -37 & 269 & 484 & 604 & -4406 \\
\phantom{}^*\mathscr{N}_r^{(3,2)} & {\bf 2} & -2 & -2 & -3 & 22 & 39 & 49  & -357 & -643 & -802& 5851 \\
\phantom{}^*\mathscr{N}_r^{(3,3)} & {\bf 3} & 2 & 3 & 4 & -29 & -52 & -65 & 474 & 854 & 1065  & -7770 \\
 \hline
\end{array}
\]
Hence, from the definition of $\hbox{}^*\mathscr{F}_{r}^{(m,j)}$, and by Theorem 3, we have
\[
\mathop{\lim}_{r\rightarrow \infty} \frac{\hbox{}^* \mathscr{F}_{3r}^{(3,1)}}{\hbox{}^* \mathscr{F}_{3r-1}^{(3,1)}}=
\mathop{\lim}_{r\rightarrow \infty} \frac{\hbox{}^* \mathscr{F}_{3r}^{(3,2)}}{\hbox{}^* \mathscr{F}_{3r-1}^{(3,2)}}=
\mathop{\lim}_{r\rightarrow \infty} \frac{\hbox{}^* \mathscr{F}_{3r}^{(3,3)}}{\hbox{}^* \mathscr{F}_{3r-1}^{(3,3)}}=2\cos{\frac{2\pi}{7}}=1.24698\ldots
\]
\[
\mathop{\lim}_{r\rightarrow \infty} \frac{\hbox{}^* \mathscr{F}_{3r-2}^{(3,1)}}{\hbox{}^* \mathscr{F}_{3r}^{(3,1)}}=
\mathop{\lim}_{r\rightarrow \infty} \frac{\hbox{}^* \mathscr{F}_{3r-2}^{(3,2)}}{\hbox{}^* \mathscr{F}_{3r}^{(3,2)}}=
\mathop{\lim}_{r\rightarrow \infty} \frac{\hbox{}^* \mathscr{F}_{3r-2}^{(3,3)}}{\hbox{}^* \mathscr{F}_{3r}^{(3,3)}}=-2\cos{\frac{4\pi}{7}}=0.445042\ldots
\]
\[
\mathop{\lim}_{r\rightarrow \infty} \frac{\hbox{}^* \mathscr{F}_{3r-1}^{(3,1)}}{\hbox{}^* \mathscr{F}_{3r-2}^{(3,1)}}=
\mathop{\lim}_{r\rightarrow \infty} \frac{\hbox{}^* \mathscr{F}_{3r-1}^{(3,2)}}{\hbox{}^* \mathscr{F}_{3r-2}^{(3,2)}}=
\mathop{\lim}_{r\rightarrow \infty} \frac{\hbox{}^* \mathscr{F}_{3r-1}^{(3,3)}}{\hbox{}^* \mathscr{F}_{3r-2}^{(3,3)}}=-2\cos{\frac{6\pi}{7}}=1.80194\ldots
\]
which is similar to the ratios of consecutive Lucas numbers or Fibonacci numbers converging to $-2\cos{\frac{4\pi}{5}}$, or their reciprocals to $2\cos{\frac{2\pi}{5}}$. We note that due to these sequence structures, the rational sequences can be replaced with the non-reduced numerator integer sequences $\phantom{}^*\mathscr{N}_r^{(3,j)}$, without affecting the ratios.

\end{example}

\begin{example}[Of cyclotomic approximation when $m=5$]
In consideration of the Euclidean distance regarding the $20$th convergent when $m=5$, we have $\left |{\bf \Psi}_{20}^{(5)}-{\bf \Phi}_5\right |$
\[
=\left |\left ( \frac{\mathscr{F}^{(5,4)}_{20}}{\mathscr{F}^{(5,5)}_{20}}, \frac{\mathscr{F}^{(5,2)}_{20}}{\mathscr{F}^{(5,4)}_{20}}, -\frac{\mathscr{F}^{(5,1)}_{20}}{\mathscr{F}^{(5,3)}_{20}}, -\frac{\mathscr{F}^{(5,3)}_{20}}{\mathscr{F}^{(5,2)}_{20}},-\frac{\mathscr{F}^{(5,5)}_{20}}{\mathscr{F}^{(5,1)}_{20}} \right )-\left ( \phi_{5\,1},\phi_{5\,2},\phi_{5\,3},\phi_{5\,4},\phi_{5\,5}\right )\right |
\]
\[
=\left |\scriptstyle{\left (\frac{42951850444254470}{25528481467235249},\frac{35685687021511133}{42951850444254470},
-\frac{4434370056070408}{15579436796165461},-\frac{46738310388496383}{35685687021511133},-\frac{25528481467235249}{13303110168211224}\right )}\right.
\]
\[
- \left.\left (2 \cos \left(\frac{2 \pi }{11}\right),2 \cos \left(\frac{4 \pi }{11}\right),2 \cos \left(\frac{6\pi }{11}\right),
2\cos \left(\frac{8 \pi }{11}\right),2 \cos \left(\frac{10\pi }{11}\right)\right )\right |
\]
\[
<10^{-13}.
\]
In terms of polynomials, this yields (to two decimal places) the remainder term
\[
\left (\scriptstyle{x^2} -{\scriptstyle{\frac{\mathscr{F}^{(5,4)}_{20}}{\mathscr{F}^{(5,5)}_{20}}}}x+1\right )\left (\scriptstyle{x^2}- {\scriptstyle{\frac{\mathscr{F}^{(5,2)}_{20}}{\mathscr{F}^{(5,4)}_{20}}}}x+1\right )\left (\scriptstyle{x^2} +{\scriptstyle{\frac{\mathscr{F}^{(5,1)}_{20}}{\mathscr{F}^{(5,3)}_{20}}}}x+1\right) \left (\scriptstyle{x^2}+{\scriptstyle{\frac{\mathscr{F}^{(5,3)}_{20}}{\mathscr{F}^{(5,2)}_{20}}}}x+1\right )
\left (\scriptstyle{x^2}+{\scriptstyle{\frac{\mathscr{F}^{(5,5)}_{20}}{\mathscr{F}^{(5,1)}_{20}}}}x+1 \right )
\]
\[
-\left (x^{10}+x^9+x^8+x^7+x^6+x^5+x^4+x^3+x^2+x+1\right )
\]
\[
=-1.48\times 10^{-13} x+1.10\times 10^{-20}x^2-2.96\times 10^{-13} x^3+2.20\times 10^{-20}
x^4-2.96\times 10^{-13} x^5
\]
\[
+2.20\times 10^{-20}x^6-
2.96\times 10^{-13} x^7+1.10\times 10^{-20}x^8-1.48\times 10^{-13} x^9.
\]
\end{example}
\begin{defn}[Of Fleck's and Weisman's congruences]
Let $p$ be a prime and $a$ be an integer. In 1913 A. Fleck discovered that
\begin{equation}
\sum_{k\equiv a\hbox{ } ({\rm mod}\hbox{ }p)}(-1)^k\binom{N}{k}\equiv 0\left ( {\rm{mod}}\hbox{ }{p^{{\left \lfloor \textstyle{\frac{N-1}{p-1}}\right \rfloor}}}\right ),
\label{eq:ap25}
\end{equation}
for all positive integers $N>0$. In 1977 C. S. Weisman \cite{weisman} extended Fleck's congruence to obtain
\begin{equation}
\sum_{k\equiv a\hbox{ } ({\rm mod}\hbox{ }p^\alpha)}(-1)^k\binom{N}{k}\equiv 0\left ( {\rm{mod}}\hbox{ }{p^{\,\omega}}\right ),\qquad \omega={{\left \lfloor \textstyle{\frac{N-p^{\alpha-1}}{\phi(p^\alpha)}}\right \rfloor}},
\label{eq:ap26}
\end{equation}
where $\alpha$, $N$ are positive integers $\geq 0$, $N\geq p^{\alpha-1}$, and $\phi$ denotes the Euler totient function. When $\alpha = 1$ it is clear that (\ref{eq:ap26}) reduces to (\ref{eq:ap25}).

We define the \emph{Fleck numbers}, $\mathfrak{F}(N, a\,\,(\bmod\,\,n))$, such that
\be
\mathfrak{F}(N, a\,\,(\bmod\,\,n))=\sum_{k\equiv a\hbox{ } ({\rm mod}\hbox{ }n)}(-1)^k\binom{N}{k}.
\label{eq:fleck}
\ee
These sums have many well known properties \cite{sun1}, \cite{sun2} such as
\be
n \,\mathfrak{F}(N,a\,\, (\bmod{\,\,n}))= \sum_{k=0}^N(-1)^k\binom{N}{k}\sum_{\gamma^n=1}\gamma^{k-a}=\sum_{\gamma^n=1}\gamma^{-a}(1-\gamma)^N,
\label{eq:unity}
\ee
from which we can deduce the recurrence relation
\be
\mathfrak{F}(N+1, a\,\, (\bmod{\,\,n}))=\mathfrak{F}(N, a\,\, (\bmod{\,\,n}))-\mathfrak{F}(N, (a-1)\,\, (\bmod{\,\,n})).
\label{eq:rec}
\ee
By modularity we also have $\mathfrak{F}(N, a\,\, (\bmod{\,\, n}))=\mathfrak{F}(N, (a+n)\,\, (\bmod{\,\, n}))$.
\end{defn}
In the following theorem we give new expressions for the renumbered Fleck numbers in terms of our polynomial functions.
\begin{thm4}
Let $r$ be a non-negative integer and $m$ be a positive integer, so that $n=2m+1$, is odd. Then the numbers in the sequences $\mathscr{F}_{-r}^{(m,j)}$ and $\mathscr{G}_{-r}^{(m,j)}$ are given by the alternating binomial sums
\begin{equation}
\mathscr{F}_{-r}^{(m,j)}=
n\sum_{a=-\infty}^\infty(-1)^{r+j+a}\binom{2r+1}{r+j+an}=n\,\mathfrak{F}(2r+1, \,(r+j)\,\, (\bmod{\,\, n})),
\end{equation}
\begin{equation}
\mathscr{G}_{-r}^{(m,j)}=
n\sum_{a=-\infty}^\infty(-1)^{r+j+1+a}\binom{2r+2}{r+j+1+an}=n\,\mathfrak{F}(2r+2, \,(r+j+1)\,\, (\bmod{\,\, n})),
\end{equation}
and so when $n$ is a prime power, they satisfy an analogue of Weisman's Congruence~\cite{weisman}.

The renumbered Fleck numbers can be written as $n\,\mathfrak{F}(2r+1, \,(r+j)\,\, (\bmod{\,\, n}))=$
\be
\sum_{t=1}^{m}\left (\mu_{m\, t} \right )^{r+1}P_{j-1}\left (\mu_{m\, t}  \right )
=\sum_{t=1}^{m}\left (\mu_{m t} \right )^{r+1}V_{j-1}\left (\phi_{m\, t}  \right )
=\sum_{t=1}^{m}\left (\mu_{m\, t} \right )^{r}\left (\phi_{m\,j t}-\phi_{m\,(j-1)t}  \right ).
\label{eq:fl1}
\ee
Similar expressions exist for the sequence terms $\mathscr{G}_{-r}^{(m,j)}$, corresponding to the renumbered Fleck numbers $\mathfrak{F}(2r+2, \,(r+j+1)\,\, (\bmod{\,\, n}))$, formed from the even rows of Pascal's triangle.
\end{thm4}

\begin{proof}[Proof of Theorem 4]
These relations are obtained by Lemma 4.2 of \cite{fleck} and the closed form expressions (\ref{eq:l36}), and (\ref{eq:l37}) of Lemma 3.2.
\end{proof}
\begin{example}
When $j=m=5$, so $n=11$, and $r=8$, we have
\[
\mathscr{F}_{-8}^{(5,5)}=11\,\sum _{a=-\infty}^{\infty} (-1)^{8+5+a} \binom{17}{8+5+11 a}=
11\,\sum _{a=-\infty}^{\infty} (-1)^{a} \binom{17}{2+11 a}
\]
\[
=\sum_{t=1}^{5}\left (\mu_{5\, t} \right )^{9}P_4\left (\mu_{5\, t}  \right )
=\sum_{t=1}^{5}\left (\mu_{5\, t} \right )^{8}\left (\phi_{5\,5 t}-\phi_{5\,4t}  \right )
=11\mathfrak{F}(17, \,2\,\, (\bmod{\,\, 11}))=-11^2. 204.
\]
\end{example}
\begin{rmk}[to Theorem 4]
Using the recurrence $Q_m(x)$ and its inverse recurrence, we can construct a similar family of sequences, which at negative indices correspond to $(2m)$ times the Fleck numbers obtained using the even modulus $(2m)$.
\end{rmk}
There are many geometric interpretations of the Fibonacci numbers and the Golden Ratio, including the triangles inscribed in a circle representation, given by J. Rigby in \cite{rigby}. In the following theorem we establish geometric relations for our rational interlacing Fibonacci sequences $\mathscr{F}_{r}^{(m,j)}$ (and $\mathscr{G}_{r}^{(m,j)}$), with the ratios of diagonal lengths between the vertices of the regular $n$-gon inscribed in the unit circle.

\begin{thm5}
Let $n$ a positive odd integer, with $n=2m+1$, and $v_{n\,0},\ldots v_{n\,(n-1)}$ the $n$ vertices of the regular $n$-gon $H_n$, inscribed in the unit circle. Let $d_{n\,k}$ be the signed distance from vertex $v_{n\,0}$ to vertex $v_{n\,k}$, so that $d_{n\,1}=2\sin{(\pi/n)}$ is the side length of $H_n$,
$d_{n\,k}=2\sin{\left ((\pi k)/n\right)}$ is the signed length of the $k$th diagonal of $H_n$, and $d_{n\,n}=0$. Then working $({\rm mod }\,\, n)$ the second subscript of $d_{n\, k}$, we have
\be
\mathscr{F}_{r}^{(m,j)}
=(-1)^{r-1}\sum_{t=1}^m \frac{d_{n\, (2j-1)t}}{d^{2r-1}_{n\, t}}
=(-1)^{r-1}\sum _{t=1}^m \frac{2 \sin \left(\frac{  (2 j-1)\pi t}{2 m+1}\right)}
{\left(2 \sin \left(\frac{\pi  t}{2m+1}\right)\right)^{2 r-1}},
\label{eq:t63}
\ee
and $\mathscr{G}_{r}^{(m,j)}=\mathscr{F}_{r}^{(m,j+1)}-\mathscr{F}_{r}^{(m,j)}$
\be
=(-1)^{r-1}\sum_{t=1}^m \frac{d_{n\, (2j+1)t}-d_{n\, (2j-1)t}}{d^{2r-1}_{n\, t}}
=(-1)^{r-1}\sum _{t=1}^m \frac{ 2 \cos \left(\frac{  2 j\pi t}{2 m+1}\right)}{\left(2 \sin \left(\frac{\pi  t}{2
   m+1}\right)\right)^{2 r-2}}.
\label{eq:t65}
\ee
Hence for $r$ a negative integer, these sums over the signed diagonal lengths correspond to integers which are the renumbered Fleck numbers multiplied by $n$.
\end{thm5}

\begin{proof}[Proof of Theorem 5]
Combining the definitions of the diagonal distances and ratios with the closed form trigonometric expressions
for $\mathscr{F}_{r}^{(m,j)}$, given in the Corollary to Theorem 2, we obtain $(\ref{eq:t63})$.

Applying the trigonometric identity $\sin{x}-\sin{y}=2\cos{\frac{(x+y)}{2}}\sin{\frac{(x-y)}{2}}$ to $\mathscr{F}_{r}^{(m,j+1)}-\mathscr{F}_{r}^{(m,j)}$, we find that $(d_{n\, (2j+1)t}-d_{n\, (2j-1)t})/d_{n\, t}=\phi_{m\, jt}$, and hence $(\ref{eq:t65})$
$\forall \,r\in\mathbb{Z}$.
\end{proof}
\begin{rmk}[To Theorem 5]
Defining the diagonal to side length ratio $r_{n\, k}$, such that $r_{n\, k}=d_{n\, k}/d_{n\, 1}$, we note that
the sums over ratios of signed diagonal lengths, corresponding to ratios in the unit $n$-gon, are closely allied to Steinbach's unsigned diagonal and ratio product formulae (see  \cite{steinbach} and \cite{lang}) given by
\[
|d_{n\, k}d_{n\, \ell}|=|d_{n\,1}|\sum_{j=0}^{\ell-1}|d_{n\,(k-\ell+2j+1)}|,\qquad
|r_{n\, k}r_{n\, \ell}|=\sum_{j=0}^{\ell-1}|r_{n\,(k-\ell+2j+1)}|,
\label{eq:t62}
\]
where we take $(k-\ell+2j+1)\,\,({\rm mod}\,\,n)$.
\end{rmk}

In the following theorem we examine the non-reduced numerator integer interlacing Fibonacci sequences of dimension $m$, $\mathscr{N}_r^{(m,j)}$, in more detail.

\begin{thm6}[Integer sequence theorem]
Let $p_1,p_2,\ldots p_t$ be all the distinct prime factors of $n=2m+1$, so that for each $p_i$ we can write $p_i=2q_i+1$. Then for $1\leq j\leq m$, and $r> 0$, we have
\[
\mathscr{N}_r^{(m,j)}=\left (\prod_{i=1}^t p_i^{\left \lfloor\frac{r-1}{q_i}\right \rfloor}\right ) \mathscr{F}_r^{(m,j)}\in\mathbb{Z}.
\]
For $\mathscr{F}_r^{(m,j)}$ with $r\leq 0$, and $n$ having at least two distinct prime factors, let $\mathscr{N}_r^{(m,j)}=\mathscr{F}_r^{(m,j)}$, and when $n=p^\alpha$ is a prime power, let
\[
\mathscr{N}_r^{(m,j)}= \left (p^\vartheta\right ) \mathscr{F}_r^{(m,j)},\qquad
\vartheta={-\alpha-\left \lfloor\frac{-2r+1-p^{\alpha-1}}{p^{\alpha}-p^{\alpha-1}}\right \rfloor}.
\]
Then with $n=p_1^{a_1}\ldots p_t^{a_t}$, we have for $r=0,-1,-2,-3,\ldots$, that the sequence terms $\mathscr{N}_r^{(m,j)}$, are the renumbered Fleck numbers, and when $n=p^\alpha$ is a prime power, the sequence terms $\mathscr{N}_r^{(m,j)}$ are the (renumbered) Fleck quotients discussed in \cite{sun2}.
\end{thm6}
\begin{corollary}
When $n=p$, a prime number, we have
\[
\left \{p^{\left \lfloor\frac{r-1}{m}\right \rfloor}\mathscr{F}_r^{(m,j)}\right \}_{r=-\infty}^{+\infty}=\left \{\mathscr{N}_r^{(m,j)}\right \}_{r=-\infty}^{+\infty},
\quad\text{\rm so that}\quad p^{\left \lfloor\frac{r-1}{m}\right \rfloor}\mathscr{F}_r^{(m,j)}\in\mathbb{Z},\,\,\,\forall\,\,\,r\in\mathbb{Z}.
\]
\end{corollary}
\begin{proof}[Proof of Theorem 6]
The case $n=p$, a prime number, was proven in Lemma 7.5 of \cite{fleck}. An expansion of this argument then leads to the deduction that
for $n=p_1^{a_1} p_2^{a_2}\ldots p_t^{a_t}$, and $r$ at positive integer values, the power of $p_i$ in the denominator of the sequence terms never exceeds $\lfloor (r-1)/q_i\rfloor$. Hence we can treat each prime number individually, replacing the denominator $m$ in the exponent of $p^{\lfloor (r-1)/m\rfloor}$, with $q_i$ for each prime factor $p_i$. It follows that multiplying together the product of the prime numbers to their individual exponents yields an integer sequence.

By Theorem 4, we know that the sequence terms $\mathscr{F}_r^{(m,j)}$ are the renumbered Fleck numbers, and so are already integer values and satisfy an analogue of Weisman's congruence.
The property relating to Fleck quotients when $n=(2m+1)=p$, a prime number, then follows from the exponent so that
\[
\left [\frac{-2r-1-p^{\alpha-1}}{p^{\alpha}-p^{\alpha-1}} \right ]_{\alpha=1}=\left [\frac{-2r-2}{p-1} \right ]
=\left [\frac{-2r-2}{2m} \right ]=\left [\frac{-r-1}{m} \right ],
\]
as required.
\end{proof}
\begin{example} We give the table for the integer interlacing Fibonacci sequence $\mathscr{N}_r^{(m,j)}$, when $m=3$, and $r\in [-6,7]$.
\[
\begin{array}{c|c|c|c|c|c|c|c|c|c|c|c|c|c|c|c|c}
j/r&  -6 & -5 & -4 & -3 & -2 & -1 & 0 & 1 & 2 & 3 & 4 & 5 & 6 &7  \\ \hline
1 & -35 & 66 & -18 & 5 & -10 & 3 & -1 & 3 & -2 & 2 & -17 & 22 & -29 & 269 \\
2& 26 & -47 & 12 & -3 & 5 & -1 & 0 & 2 & -3 & 4 & -37 & 49 & -65 & 604  \\
3&-13 & 22 & -5 & 1 & -1 & 0 & 0 & 1 & -2 & 3 & -29 & 39 & -52 & 484   \\
\end{array}
\]
\end{example}
\begin{rmk}[to Theorem 6]
When $n=2m+1$ has repeated prime factors, it appears to be the case that particular sequences out of the $m$ sequences have further divisibility properties. One area for future investigation would be to derive a concise formula that describes these patterns.
\end{rmk}
\begin{lemma}[Sum of diagonal products lemma]
Let $H_n$ be the regular odd-sided unit $n$-gon $H_n$, as defined in Theorem 5, so that $n=2m+1>0$, $H_n$ has signed diagonal distances given by $d_{n\,k}=2\sin{\left ((\pi k)/n\right)}$. Then we have
\[
\sum_{j=1}^m  d_{n\,(2j-1)u}\,d_{n\,(2j-1)v}
=\begin{cases}
0 & \text{if}\,\,\,u\neq v\\
2m+1 &  \text{if}\,\,\,u= v\\
\end{cases}
\]
\end{lemma}
\begin{proof} Applying trigonometric identities we can write the sum as a difference of sine ratios, such that
\[
\sum _{j=1}^m 2 \sin \left(\frac{\pi  u (2 j-1)}{2 m+1}\right) 2 \sin \left(\frac{\pi  v (2 j-1)}{2 m+1}\right)
=\frac{\sin \left(\frac{2 m \pi  (u-v)}{1+2 m}\right)}{\sin \left(\frac{\pi  (u-v)}{1+2 m}\right)}
-\frac{\sin \left(\frac{2 m   \pi  (u+v)}{1+2 m}\right)}{\sin \left(\frac{\pi  (u+v)}{1+2 m}\right)},
\]
thus eliminating the $j$ variable. Using $\sin{(\pi-x)}=-\sin{x}$, we see that the two terms cancel out if $u\neq v$,  and when $u=v$, we deduce the result via the identity $\sum_{k=0}^N\sin^2(kx)=1/4(1+2N-\csc{x}\sin{[x(1+2N)]})$.
\end{proof}

\begin{thm7}
Let $n=2m+1$ be a positive integer. Then the sequence terms $\mathscr{F}_r^{(m,j)}$ and $\mathscr{G}_r^{(m,j)}$, obey the relations
\[
\mathscr{F}_{2r}^{(m,1)}
=\frac{-1}{n}\sum_{j=1}^m\left ( \mathscr{F}_r^{(m,j)}\right )^2,\qquad
\mathscr{F}_{2r+1}^{(m,1)}=\frac{1}{2n}\left (\left (\mathscr{G}_r^{(m,0)}\right )^2+2\sum_{j=1}^m\left ( \mathscr{G}_r^{(m,j)}\right )^2\right ),
\]
so that for all $r\in \mathbb{Z}$ we have
\[
\sum_{t=1}^{m}\left (\mu_{m\, t} \right )^{-2r+1}=
 \frac{-1}{n}\sum_{j=1}^m\left (\sum_{t=1}^{m}\left (\mu_{m\, t} \right )^{-r}\left (\phi_{m\, j t}-\phi_{m\,(j-1)t}\right )\right )^2,
\]
and $\sum_{t=1}^{m}\left (\mu_{m\, t} \right )^{-2r}$
\[
=\frac{1}{2n}\left (4\left (\mu_{m\, t} \right )^{-2r+2}+
2\sum_{j=1}^m\left (\sum_{t=1}^{m}\left (\mu_{m\, t} \right )^{-r}\left (\phi_{m\, (j+1) t}-2\phi_{m\, j t}+\phi_{m\,(j-1)t}\right )\right )^2\right ).
\]
Hence each of the sequence terms $\mathscr{F}_r^{(m,j)}$ and $\mathscr{G}_r^{(m,j)}$ can be written as an integer linear combination of sums of square integers.
\end{thm7}
\begin{corollary}
Let $n=2m+1>0$. Then we have the Fleck number relations
\[
\mathfrak{F}(4r+1, \,(2r+1)\,\, (\bmod{\,\, n}))=-\sum_{j=1}^m\left (\mathfrak{F}(2r+1, \,(r+j)\,\, (\bmod{\,\, n}))\right )^2,
\]
\[2\mathfrak{F}(4r+3, \,(2r+2)\,\, (\bmod{\,\, n}))=\mathfrak{F}(4r+4, \,(2r+2)\,\, (\bmod{\,\, n}))\]
\[
=\left (\mathfrak{F}(2r+2, \,(r+1)\,\, (\bmod{\,\, n}))\right )^2+2\sum_{j=1}^m\left (\mathfrak{F}(2r+2, \,(r+j+1)\,\, (\bmod{\,\, n}))\right )^2.
\]
The corresponding combinatorial identities (not given in \cite{gould}) are
\[
\sum _{a=-\infty}^{\infty}{\textstyle{ (-1)^a \binom{4 r+1}{2 r+1+a n}}}
=\sum _{j=1}^m \left(\sum _{a=-\infty}^{\infty}{\textstyle{ (-1)^{r+j+a} \binom{2 r+1}{r+j+a n}}}\right)^2,
\]
\[
2 \sum _{a=-\infty}^{\infty}{\textstyle{ (-1)^a \binom{4 r+3}{2 r+2+a n}}}
=\sum _{a=-\infty}^{\infty}{\textstyle{ (-1)^a \binom{4 r+4}{2 r+2+a n}}}
\]
\[
=\left(\sum _{a=-\infty}^{\infty} {\textstyle{(-1)^{r+1+a} \binom{2 r+2}{r+1+a n}}}\right)^2+
2 \sum _{j=1}^m \left(\sum _{a=-\infty}^{\infty}
{\textstyle{ (-1)^{r+j+1+a} \binom{2r+2}{r+j+1+a n}}}\right)^2.
\]
\end{corollary}
We note that the above identities appear to be true when $n=2m$.
\begin{proof}
Writing the first display of the Theorem in terms of the signed diagonals of the unit $n$-gon $H_n$ and rearranging, gives us
\[
\sum_{u=1}^m \sum_{j=1}^m \left ( d_{n\,u}\right )^{-4r+2}\left (d_{n\,(2j-1)u}\right )^2
\]
\[
+2\sum_{u=1}^m \sum_{v=u+1}^m \left ( d_{n\,u}\right )^{-2r+1}\left ( d_{n\,v}\right )^{-2r+1}
\sum_{j=1}^m d_{n\,(2j-1)u}\,d_{n\,(2j-1)v}.
\]
By Lemma 3.4, for $u\neq v$ the inner sum is zero, and by (\ref{eq:l36}) the result follows. The second display can be obtained similarly.

To see that each sequence term can be written as an integer linear combination of sums of square integers, by the preceding results we can write each term
$\mathscr{F}_{2r}^{(m,1)}$ as a sum of $m$ squares, and each term $\mathscr{F}_{2r+1}^{(m,1)}$ as a sum of $m$ squares plus twice a square. By (\ref{eq:s12}) and (\ref{eq:s13}), each of the sequence terms $\mathscr{F}_r^{(m,j)}$ and $\mathscr{G}_r^{(m,j)}$, can be written as an integer linear combination of the $\mathscr{F}_r^{(m,1)}$, and hence the result.
\end{proof}

\begin{example}[Of sum of squares representation]
In consideration of the non-reduced numerator $\mathscr{N}_{10}^{(6,3)}$, of the sequence term $\mathscr{F}_{10}^{(6,3)}$, by (\ref{eq:s12}) can write
\[
-\mathscr{N}_{10}^{(6,3)}=29226191=-5\mathscr{N}_{10}^{(6,1)}-5\mathscr{N}_{9}^{(6,1)}-\mathscr{N}_{8}^{(6,1)}
\]
\[
=5\times 7480420-5\times 1713705+ 392616
\]
\[
= 5\times(1505^2 +1421^2 +1245^2 + 1010^2 +    702^2+ 365^2 )
\]
\[
-5\times (702^2 + 645^2 + 543^2 + 411^2 + (2\times 365^2) +  260^2 + 84^2 )
\]
\[
+344^2  +327^2 +  283^2 + 234^2 +  159^2+85^2,
\]
so that after simplification, this method enables $\mathscr{N}_{10}^{(6,3)}$ to be written as an integer linear combination of 16 square integers.
\end{example}

\section{Properties of the Polynomial Generating Functions}

\begin{thm8}
With $P_0(x)=1$, $P_1(x)=1+x/3$, $Q_1(x)=1+x/2$ and $Q_2(x)=x^2/4+x+1/2$, the polynomials $P_m(x)$ and $Q_m(x)$ respectively satisfy the three-term recurrences
\be
P_{m+1}(x)=(x+2)P_{m}(x)-P_{m-1}(x),
\label{eq:th21}
\ee
\be
Q_{m+1}(x)=(x+2)Q_m(x)-Q_{m-1}(x),
\label{eq:th22}
\ee
the ordinary differential equations
\be
x(4+x)P_m''(x)+2(x+3)P_m'(x)-m(m+1)P_m(x)=0,
\label{eq:th23}
\ee
\be
x(4+x)Q_m''(x)+(2+x)Q_m'(x)+m^2Q_m(x)=0,
\label{eq:th24}
\ee
and for integers $k,\ell$, the explicit orthogonality condition
\be
\int_{-4}^0 P_\ell(x)P_k(x){x^{1/2} \over {(4+x)^{1/2}}}dx=2\pi i \delta_{\ell\, k}, ~~~~~~\ell,k\geq 0
\label{eq:th25}
\ee
\be
\int_{-4}^0 {{Q_\ell(x)Q_k(x)} \over {x^{1/2}(4+x)^{1/2}}}dx=-2\pi i\delta_{\ell\, k}, ~~~~~~k \neq 0,
\label{eq:th26}
\ee
where $\delta_{\ell\,k}$ is the Kronecker symbol.

Let
\be
M_m^P(s) \equiv \int_{-4}^0 {P_m(x)x^{s-3/4} \over {(4+x)^{3/4}}}dx, ~~~~~~\mbox{Re} ~s>-1/4,
\label{eq:th27}
\ee
\be
M_m^Q(s) \equiv \int_{-4}^0{{Q_m(x)x^{s-5/4}} \over {(4+x)^{3/4}}}dx.
\label{eq:th28}
\ee
Then up to normalization, these Mellin transforms have the form
\be
M_m^P(s)=(-1)^{s+5/4}4^s 4^{-m-1} \Gamma(1/4) p_m(s) {{\Gamma\left(s+{1 \over 4}\right)} \over {\Gamma\left(s+{{2m+1} \over 2}\right)}},
\label{eq:th29}
\ee
\be
M_m^Q(s)=(-1)^{s+3/4}4^{s-1} \Gamma(5/4) q_m(s) {{\Gamma\left(s-{1 \over 4}\right)} \over {\Gamma(s+m)}}.
\label{eq:th210}
\ee
\end{thm8}
\begin{corollary}
Closed form expressions for $P_m(x-2)$ and $Q_m(x-2)$ are given by $P_m(x-2)= $
\be
2^{-m-1}\left (\left(1-\frac{\sqrt{x+2}}{\sqrt{x-2}}\right) \left(x-\sqrt{x^2-4}\right)^m+\left(1+\frac{\sqrt{x+2}}{\sqrt{x-2}}\right)
   \left(x+\sqrt{x^2-4}\right)^m\right ),
\label{eq:th214}
\ee
and
\be
Q_m(x-2)=2^{-m} \left(\left(x-\sqrt{x^2-4}\right)^m+\left(x+\sqrt{x^2-4}\right)^m\right),
\label{eq:th215}
\ee
where for $r\leq [m/2]$ we also have
\be
P_m(x)=\sum_{j=0}^r (-1)^r\binom{r}{j}(x+2)^{r-j}P_{m-r-j}(x),
\label{eq:th216}
\ee
and
\be
Q_m(x)=\sum_{j=0}^r (-1)^r\binom{r}{j}(x+2)^{r-j}Q_{m-r-j}(x).
\label{eq:th217}
\ee
We have the orthogonal polynomial relations
\be
\int_{-4}^0 x^{1/2}(4+x)^{-1/2}P_k(x)(\mbox{any polynomial of degree}~<k)dx=0.
\label{eq:th211}
\ee
\be
\int_{-4}^0 x^{-1/2}(4+x)^{-1/2}Q_k(x)(\mbox{any polynomial of degree}~<k)dx=0.
\label{eq:th212}
\ee
The polynomial factors of $M_m^P(s)$ and $M_m^Q(s)$ satisfy the functional equations
\be
p_n(s)=\pm p_n(1-s),\qquad q_n(s)=\pm q_n(1-s)
\label{eq:th213}
\ee
and have zeros only on the critical line Re $s=1/2$.

\end{corollary}

\begin{proof}[Proof of Theorem 8]
The three term recurrences for $P_m(x)$ and $Q_m(x)$ in $(\ref{eq:th21})$ and $(\ref{eq:th22})$ follow from the Legendre function expression for $P_m(x)$ given in $(\ref{eq:th81})$, and using \cite{andrews} [p. 99 or 247 or 295] with $\alpha=\beta=-1/2$ the Jacobi polynomial $P_n^{(\alpha,\beta)}(x)$ relation
$$Q_m(x)={{(2m)(m-1)!} \over {(1/2)_m}}P_m^{(-1/2,-1/2)}\left(1+{x \over 2}\right).$$
We sketch the details for $Q_m(x)$. A recurrence for the polynomial $P_n^{(-1/2,-1/2)}(z)$ is
$$(n+1)n(2n-1)P_{n+1}^{(-1/2,-1/2)}(z)$$
$$=n(2n-1)(2n+1)zP_n^{(-1/2,-1/2)}(z)-(n-1/2)^2(2n+1)P_{n-1}^{(-1/2,-1/2)}(z),$$
and using the change of variable $z=1+x/2$ and (2.5.14) of \cite{andrews}, applied to the polynomials $Q_m(x)$, we obtain $(\ref{eq:th22})$.

Considering now the ordinary differential equation satisfied by $P_\nu^\mu(z)$,
$$(1-z^2){{d^2u} \over {dz^2}}-2z{{du} \over {dz}}+\left[\nu(\nu+1)-{\mu^2 \over {1-z^2}}\right]u=0,$$
and an elementary application of the chain rule, we find
$$x(4+x)P_m''(x)+2(x+3)P_m'(x)-m(m+1)P_m(x)=0,$$
and hence $(\ref{eq:th23})$.

By using an integrating factor $x^{3/2}(4+x)^{1/2}$, the differential equation for $P_m(x)$
may be written as
$$x^{3/2}(4+x)^{1/2}P_m''(x)+2(x+3)x^{1/2}(4+x)^{-1/2}P_m'(x)-m(m+1)x^{1/2}(4+x)^{-1/2}P_m(x)=0.$$
We then obtain
$${d \over {dx}}[x^{3/2}(4+x)^{1/2}P_m'(x)]=m(m+1)x^{1/2}(4+x)^{-1/2}P_m(x).$$
Writing this equation for $P_k(x)$, multiplying the $P_m(x)$ equation by $P_k(x)$, and the
$P_k(x)$ equation by $P_m(x)$ and subtracting there follows
$$P_k(x){d \over {dx}}[x^{3/2}(4+x)^{1/2}P_m'(x)]-P_m(x){d \over {dx}}[x^{3/2}(4+x)^{1/2}P_k'(x)]$$
$$=[m(m+1)-k(k+1)]x^{1/2}(4+x)^{-1/2}]P_k(x)P_m(x).$$
Thus
$${d \over {dx}}\left\{x^{3/2}(4+x)^{1/2}[P_k(x)P_m'(x)-P_m(x)P_k'(x)]\right\}$$
$$=[m(m+1)-k(k+1)]x^{1/2}(4+x)^{-1/2}]P_k(x)P_m(x).$$
By integrating between $x=-4$ and $0$, we obtain the stated result $(\ref{eq:th25})$.

The orthogonality of the sequence $\{P_m(x)\}_{m \geq 0}$ allows the development of integral transforms
with zeros only along vertical lines in the complex plane and hence $(\ref{eq:th27})$ and $(\ref{eq:th29})$.

For $(\ref{eq:th24})$ we use the hyperbolic trigonometric function analogue for Chebyshev polynomials to obtain
$$Q_m(x)=2\cosh\left(2m\sinh^{-1}\left({\sqrt{x} \over 2}\right)\right)
= 2\,{}_2\mathrm{F}_1\left(-m,m;{1 \over 2};-{x \over 4}\right).$$
The differential equation for the Gauss hypergeometric function ${}_2\mathrm{F}_1(a,b;c;z)$,
$$z(1-z){{d^2y} \over {dz^2}}+[c-(a+b+1)z]{{dy} \over {dz}}-aby=0,$$
becomes for $Q_m(x)$
$${{d^2y} \over {dx^2}}+{{(2+x)} \over {x(4+x)}}{{dy} \over {dx}}+{m^2 \over {x(4+x)}}y=0,$$
and the result follows.

In consequence, the family $\{Q_m(x)\}_{m \geq 1}$ is orthogonal, and
with the integrating factor $\sqrt{x}\sqrt{4+x}$, the differential equation may be written as
$${d \over {dx}}\left (\sqrt{x}\sqrt{4+x}{{dy} \over {dx}}\right)={m^2 \over {\sqrt{x}\sqrt{4+x}}}y.$$
The integrating factor is obtained as the exponential of
$$\int {{(2+x)} \over {x(4+x)}}dx={1 \over 2}\ln[x(4+x)].$$
We then obtain the orthogonality relation $(\ref{eq:th26})$ [the steps being omitted]
$$\int_{-4}^0 {{Q_m(x)Q_k(x)} \over {x^{1/2}(4+x)^{1/2}}}dx=-{2\pi i}\delta_{mk}, ~~~~~~k \neq 0.$$
Accordingly, we have a (generalised) Mellin transform
$$M_m^Q(s) \equiv \int_{-4}^0{{Q_m(x)x^{s-5/4}} \over {(4+x)^{3/4}}}dx,$$
of the form $(\ref{eq:th210})$, so that
$$M_m^Q(s)=(-1)^{s+3/4}4^{s-1} \Gamma(5/4) q_m(s) {{\Gamma\left(s-{1 \over 4}\right)} \over {\Gamma(s+m)}}.$$

In the Corollary, $(\ref{eq:th214})$ and $(\ref{eq:th215})$ are obtained by solving the recurrences in $(\ref{eq:th21})$ $(\ref{eq:th21})$, whereas
$(\ref{eq:th216})$ and $(\ref{eq:th217})$ arise from iteratively applying the recurrences.

The last part of the Corollary follows from properties of orthogonal polynomials.

The proof that the polynomial factors of $M_m^P(s)$ and $M_m^Q(s)$ satisfy the functional equations $p_n(s)=\pm p_n(1-s)$,
and $q_n(s)=\pm q_n(1-s)$, and have zeros only on the critical line Re $s=1/2$, follows that given in \cite{mwcmcl}

\end{proof}

\begin{thm9}
The polynomials $P_m(x)$ obey the Christoffel-Darboux formula 
\be
\sum_{m=0}^n P_m(y)P_m(x)={{P_{n+1}(x)P_n(y)-P_{n+1}(y)P_n(x)} \over {x-y}},
\label{eq:t72}
\ee
the confluent form of which (i.e., for $y \to x$) is
\be
\sum_{k=0}^n P_k^2(x)=P'_{n+1}(x)P_n(x)-P_{n+1}(x)P_n'(x).
\label{eq:t73}
\ee
They also satisfy the relation
\be
{{dP_m(x)} \over {dx}}={1 \over {x(4+x)}}\left\{-(2m-1)P_{m-1}(x)+[m(x+2)-1]P_m(x)\right\},
\label{eq:t74}
\ee
and have the raising and lowering operators
\be
R_m=x(x+4){{d} \over {dx}} + m(x+2) +x+3,
\label{eq:t75}
\ee
\be
L_m= -x(x+4){{d} \over {dx}} + m(x+2) -1,
\label{eq:t76}
\ee
such that $R_m P_m(x)=(2m+3)P_{m+1}(x)$, and $L_m P_m(x)=(2m-1)P_{m-1}(x)$. The ordinary differential equation for $P_m(x)$ can then be written in terms of these operators.

The polynomials $Q_m(x)$ obey the quadratic identity
\be
Q_m^2(x)=Q_{2m}(x)+2,
\label{eq:t77}
\ee
have the generating function
\be
\sum_{m=0}^\infty {{(1/2)_m} \over {m!}}Q_m(x)r^m=R^{-1}(1-r+R)^{1/2}(1+r+R)^{1/2},
\label{eq:t78}
\ee
where $R=(1-2r-x r +r^2)^{1/2}$, and
satisfy the differential relation
\be
{d \over {dx}}Q_m(x)={{2m\sin[m\cos^{-1}(1+x/2)]} \over {\sqrt{-x(4+x)}}}.
\label{eq:t79}
\ee
\end{thm9}

\begin{proof}[Proof of Theorem 9]
The polynomials $P_m(x)$ have the hypergeometric form
$$P_m(x)={}_2\mathrm{F}_1\left(-m,m+1;{3 \over 2};-{x \over 4}\right)
={2 \over {(2m+1)\sqrt{x}}} \sinh\left[(2m+1)\sinh^{-1}\left({\sqrt{x} \over 2}\right)\right].$$
Hence their ODE may also be found from that of the $_2\mathrm{F}_1$ function.

If we normalize such that $P_\ell(x) \to P_\ell(x)/\sqrt{2\pi i}$, so that
$$\int_{-4}^0 P_k^2(x){x^{1/2} \over {(4+x)^{1/2}}}dx=1,$$
we obtain the Christoffel-Darboux formula of $(\ref{eq:t72})$ 
$$\sum_{m=0}^n P_m(y)P_m(x)={{P_{n+1}(x)P_n(y)-P_{n+1}(y)P_n(x)} \over {x-y}},$$
and the confluent form of this result $(\ref{eq:t73})$
$$\sum_{k=0}^n P_k^2(x)=P'_{n+1}(x)P_n(x)-P_{n+1}(x)P_n'(x),$$
then follows.

When the relation
$$(1-z^2){{dP_m^{-1/2}(z)} \over {dz}}=\left(m-{1 \over 2}\right)P_{m-1}^{-1/2}(z)-m z P_m^{-1/2}(z)$$
is transformed to the polynomials $P_m(x)$, the result is
$${{dP_m(x)} \over {dx}}={1 \over {x(4+x)}}\left\{-(2m-1)P_{m-1}(x)+[m(x+2)-1]P_m(x)\right\},$$
which is $(\ref{eq:t74})$. The raising and lowering operators of $(\ref{eq:t75})$ and $(\ref{eq:t76})$
can then be deduced.

From the application of linear and quadratic transformation of the $_2\mathrm{F}_1$ function we have the following.
$$Q_m(x)=2 \left(1+{x \over 4}\right)^m {}_2\mathrm{F}_1\left(-m,{1 \over 2}-m;{1 \over 2};{x \over {x+4}}\right)
$$
$$={{(4+x)^m} \over {2^{2m}}}\left[\left(1-\sqrt{x \over {4+x}}\right)^{2m}+\left(1+\sqrt{x \over {4+x}}\right)^{2m}\right]$$
$$=2 \left(1+{x \over 4}\right)^{-m} {}_2\mathrm{F}_1\left(m,{1 \over 2}+m;{1 \over 2};{x \over {x+4}}\right)$$
$$=2 \left(1+{x \over 4}\right)^{1/2} {}_2\mathrm{F}_1\left({1 \over 2}+m,{1 \over 2}-m;{1 \over 2};-{x \over 4}\right)=2\cosh\left(2m\sinh^{-1}\left({\sqrt{x} \over 2}\right)\right).$$

As
$$Q_m(x)=2\,\, {}_2\mathrm{F}_1\left(-{m \over 2},{m \over 2}; {1 \over 2};-{x \over 4}(4+x)\right),$$
we find that $Q_m(x)=Q_{m/2}[x(4+x)]$. As
$$Q_m(x)=2(-1)^m {}_2\mathrm{F}_1\left(-m,m;{1 \over 2};1+x/4\right),$$
we determine that $Q_m(x)=(-1)^m Q_m(-4-x)$.  Furthermore,
$$Q_m(x)={{2\sqrt{\pi} (2m-1)!} \over {(m-1)!\Gamma(m+1/2)}}\left(1+{x \over 4}\right)^m
{}_2\mathrm{F}_1\left(-m,{1 \over 2}-m;1-2m;{1 \over {1+x/4}}\right)$$
$$={{2\sqrt{\pi} (2m-1)!} \over {(m-1)!\Gamma(m+1/2)}}\left({x \over 4}\right)^m
{}_2\mathrm{F}_1\left(-m,{1 \over 2}-m;1-2m;-{4 \over x}\right).$$

With $(a)_n$ the Pochhammer symbol, we note the limit for $j>0$
$$\lim_{a \to -b}{{(a+b)_j} \over {(2a+2b)_j}}={1 \over 2},$$
otherwise this ratio is $1$ for $j=0$.
We then obtain a reduction of Clausen's identity for the square of a special $_2\mathrm{F}_1$ function,
$Q_m^2(x)=Q_{2m}(x)+2,$
which is $(\ref{eq:t77})$.

To see  $(\ref{eq:t78})$ we identify $Q_m(x)$ in terms of Jacobi polynomials $P_n^{(\alpha,\beta)}(x)$.  We use \cite{andrews}
[p. 99 or 247 or 295] with $\alpha=\beta=-1/2$ and obtain
$$Q_m(x)={2(m!) \over {(1/2)_m}}P_m^{(-1/2,-1/2)}\left(1+{x \over 2}\right),$$
which is the Gegenbauer polynomial case of $C_m^{\lambda \to 0}$.

A generating function for Jacobi polynomials is \cite{andrews} (p. 298) 
$$F(z,r)=\sum_{n=0}^\infty P_n^{(\alpha,\beta)}(z)r^n=2^{\alpha+\beta}R^{-1}(1-r+R)^{-\alpha}(1+r+R)^{-\beta},$$
where $R=(1-2zr+r^2)^{1/2}$.  Correspondingly we find the generating function
$$\sum_{m=0}^\infty {{(1/2)_m} \over m!}Q_m(x)r^m=R^{-1}(1-r+R)^{1/2}(1+r+R)^{1/2},$$
as required, where now $R=(1-2r-x r +r^2)^{1/2}$.

By \cite{andrews} (p. 297)
$${d \over {dx}}P_n^{(-1/2,-1/2)}(x)={n \over 2}P_{n-1}^{(1/2,1/2)}(x)
={{\Gamma(n+1/2)\sin(n\cos^{-1} x)} \over {\sqrt{\pi}(n-1)!\sqrt{1-x^2}}}.$$
We then obtain
$${d \over {dx}}Q_m(x)={{2m\sin[m\cos^{-1}(1+x/2)]} \over {\sqrt{-x(4+x)}}},$$
which is $(\ref{eq:t79})$.

We note that in terms of Jacobi polynomials $P_m(x)$ can be written as
$$P_m(x)={{m!} \over {(3/2)_m}} P_m^{(1/2,-1/2)}\left(1+{x \over 2}\right).$$

\end{proof}

\begin{rmk}[to Theorem 9 - Generalized raising operator and Rodrigues' formula]

It is possible to obtain a generalised Rodrigues' formula for the polynomials $P_m(x)$, as we now present.
Following the procedure of \cite{coffey2015} we put
$R_m=f_1{d \over {dx}}g_2+h$, where $h$ is an arbitrary function and $f_1$ and $g_2$ are functions to be
determined.  We find
$$g_2(x)=\exp\left[-\int {{h(x)} \over {x(x+4)}}dx\right]x^{(m+3/2)/2}(x+4)^{(m+1/2)/2},$$
and
$$f_1(x)={{x(x+4)} \over {g_2(x)}}=\exp\left[\int {{h(x)} \over {x(x+4)}}dx\right]x^{-m/2+1/4}(x+4)^{-m/2+3/4}.$$
By way of the iteration $P_{m+1}(x)={1 \over {(2m+3)}}{1 \over {(2m+1)}}\cdots {1 \over 5}\cdot {1 \over 3}\cdot
{1 \over 1}R_mR_{m-1}\cdots R_1R_0P_0(x)$
for $h=0$ we obtain a generalised Rodrigues' formula
$$P_{m+1}(x)={1 \over {(2m+3)!!}}x^{-m/2+3/4}(x+4)^{-m/2+5/4}{d \over {dx}}\left(x^{3/2}(x+4)^{3/2}{d \over {dx}}
\right)^{m-1} x^{3/4}(4+x)^{1/4},$$
where $(2n+1)!!=(2n+1)(2n-1)\cdots 3\cdot 1$.
\end{rmk}

\begin{thm10}
Let $C_n(x)$ be the minimal polynomial of $2\cos{\left (\frac{\pi}{n}\right )}$, and $\Theta_n(x)$ the minimal polynomial of $2\cos{\left (\frac{2\pi}{n}\right )}$. Then we have

\begin{align}
\frac{P_m(x-2)}{\Theta_{2m+1}(x)}=&
\begin{cases} =1 & \,\,\,\text{\rm if}\,\,\, 2m+1\,\,\,\text{\rm is a prime number},\\
\in \mathbb{Z}[x] & \,\,\,\text{ \rm otherwise.}
\end{cases}
\label{eq:s75}\\
\frac{Q_m(x-2)}{C_{2m}(x)}=&
\begin{cases} =1 & \,\,\,\text{\rm if $m$ is a power of 2},\\
=x & \,\,\,\text{\rm if $m$ is an odd prime number},\\
\in \mathbb{Z}[x] & \,\,\,\text{ \rm otherwise.}
\end{cases}
\label{eq:s76}\\
\frac{-\,Q_{2m+1}(-x-2)}{x C_{4m+2}(x)}=\frac{(-1)^m P_m(-x^2)}{C_{4m+2}(x)}=&
\begin{cases} =1 & \,\,\,\text{\rm if}\,\,\, 2m+1\,\,\,\text{\rm is a prime number},\\
\in \mathbb{Z}[x] & \,\,\,\text{ \rm otherwise.}
\end{cases}
\label{eq:s77}\\
\frac{Q_m(-x-2)}{C_{4m}(x)}=\frac{(-1)^m Q_m(-x^2)}{C_{4m}(x)}=&
\begin{cases} =1 & \,\,\,\text{\rm if $m$ is a power of 2},\\
\in \mathbb{Z}[x] & \,\,\,\text{ \rm otherwise.}
\end{cases}
\label{eq:s78}\\
\frac{(-1)^m P_m(-x-2)}{C_{2m+1}(x)}=\frac{V_m(x)}{C_{2m+1}(x)}=&
\begin{cases} =1 & \,\,\,\text{\rm if $2m+1$ is a prime number},\\
\in \mathbb{Z}[x] & \,\,\,\text{ \rm otherwise.}
\end{cases}
\label{eq:s79}
\end{align}
\end{thm10}
\begin{corollary}
Let $\rho_n(x)$ be the minimal polynomial of $2\cos{\left (\frac{2\pi}{n}\right )}-2$, $\tau_n(x)$ be the minimal polynomial of $2\cos{\left (\frac{\pi}{n}\right )}-2$, and $\varphi_n(x)$ be the minimal polynomial of $-2\cos{\left (\frac{2\pi}{n}\right )}-2$. Then
\be
P_m(x)=\mathop{\prod_{d\mid 2m+1}}_{d\geq 3}  \rho_d(x),\qquad Q_m(x)=\mathop{\prod_{d\mid m}}_{m/d\,\,\text{\rm is odd}} \tau_{2d}(x),
\qquad\mathcal{P}_m(x)= \mathop{\prod_{d\mid 2m+1}}_{d\geq 3} \varphi_d(x),
\label{eq:T21}
\ee
and
\be
\mathcal{Q}_m(x)=
P_{m_1}(x)\mathcal{P}_{m_1}(x)=\mathop{\prod_{d\mid 2m_1+1}}_{d\geq 3} \rho_d(x)\varphi_d(x),
\label{eq:T22}
\ee
when $m=2m_1$ is even. If $m=2m_1+1=2^{r}(m_{r}+1)-1$ is odd, with $m_{r}=2m_{r+1}$ even, so that
for $1\leq j \leq r-1$, $m_j=2m_{j+1}+1$, we have
\be
\frac{m_1+1}{m_2+1}=\frac{m_2+1}{m_3+1}=\ldots = \frac{m_{r-1}+1}{m_{r}+1}=2,
\label{eq:T24}
\ee
then
\be
\mathcal{Q}_m(x)=
\mathcal{Q}_{m_{r}}(x)\prod_{j=1}^{r-1} Q_{m_j+1}(x)=P_{m_{r+1}}(x)\mathcal{P}_{m_{r+1}}(x)\prod_{j=1}^{r-1} Q_{m_j+1}(x),
\label{eq:T23}
\ee
\end{corollary}
\begin{proof}[Proof of Theorem 10]

We recall that the minimal polynomial of an algebraic number $\beta$, is defined to be the monic polynomial of minimal degree, with rational coefficients, which has $\beta$ as one of its roots. Such polynomials often exhibit structural properties, such as $\Phi_n(x)$, the minimal polynomial of a primitive $n$th root of unity, $e(k/n)$, with $(k,n)=1$, which satisfies
\be
x^n-1=\prod_{d\mid n}\Phi_d(x).
\label{eq:p15}
\ee
It was shown in \cite{mccombs} that when $n$ is a prime number $p$, then the minimal polynomial $\Theta_n(x)$ of $2\cos{(2\pi/n)}$ is given by $\Theta_n(x)=f_n(x)$, with
\[
f_n(x)=\sum _{k=0}^{[\frac{n}{2}]} (-1)^k \binom{n-k}{k} x^{n-2 k}-\sum _{k=0}^{[\frac{n+1}{2}]} (-1)^k \binom{n-k}{k-1} x^{n-2 k+1},
\]
and that $\Theta_n(x)\mid f_n(x)$ for all $n\in \mathbb{N}$. By algebraic manipulation we have $P_m(x-2)=f_m(x)$, and hence $(\ref{eq:s75})$. In fact, for $p$ a prime number, we can write $\Theta_p(2x)=2^{(p-1)/2}\Psi_p(x)$, where$\Psi_n(x)$ denotes the minimal polynomial of the algebraic number $\beta(n)=\cos{(2 \pi/n)}$.

It was shown by Watkins and Zeitlin \cite{watkins} that analogous formulae to $(\ref{eq:p15})$ for $\Psi_n(x)$ are given by
\be
T_{n_1+1}(x)-T_{n_1}(x)=2^{n_1}\prod_{d\mid n}\Psi_d(x),\qquad n=2n_1 +1\,\,\,\text{is odd},
\label{eq:p1}
\ee
\be
T_{n_1+1}(x)-T_{n_1-1}(x)=2^{n_1}\prod_{d\mid n}\Psi_d(x),\qquad n=2n_1\,\,\,\text{is even},
\label{eq:p101}
\ee
from which we can establish the explicit formula
\[
\Psi_n(x)=\mathop{\prod_{k=1}^{[n/2]}}_{(n,k)=1}\left (x-\cos\left ( \frac{2\pi k}{n}\right )\right ),
\]
so that deg $\Psi_n(x)=1$ if $n=1,2$ and $\phi(n)/2$ if $n\geq 3$. From this one deduces that $C_n(x)$, the minimal polynomial of $2\cos{\pi/n}$, is given by
\[
C_1=2\Psi_2\left (\frac{x}{2}\right ),\qquad C_n(x)=2^{\phi(2n)/2}\Psi_{2n}\left (\frac{x}{2}\right ),\qquad n\geq 2.
\]
It follows that deg $C_n(x)=1$ if $n=1$ and $\phi(2n)/2$ if $n\geq 2$, the zeros of $C_n(x), n\geq 2$, are $2 \cos(\pi k/n)$, with $k=1,...,n-1$ and $(k,2n)=1$. Hence each of the expressions in $(\ref{eq:s76})$, $(\ref{eq:s77})$, $(\ref{eq:s78})$ and $(\ref{eq:s79})$ are in $\mathbb{Z}[x]$ and equal to 1 for the prime conditions stated. In fact, for $n$ an odd integer, we have $\Theta_n(-x)=(-1)^{\phi{(2n)}/2} C_n(x)$.

When $n$ is a power of 2 we use the identity $2^{2^{m-2}}\Psi_{2^m}(x)=2\,T_{2^{m-2}}(x)$,
and the result follows.

For completeness we state the case that $n$ is an odd prime power $p^m$, with $p=2q+1$, for which we have
\[
2^{p^{m-1}(p-1)/2}\Psi_{p^m}(x)=2\,\left (\sum_{j=1}^{q} T_{p^{m-1} j}(x)\right ) +1.
\]
The Corollary can be deduced either from the relations $(\ref{eq:s75}), (\ref{eq:s76})$, $(\ref{eq:s77})$, $(\ref{eq:s78})$ and $(\ref{eq:s79})$, with $x-2$ replaced by $x$, in conjunction with
$(\ref{eq:s725})$, $(\ref{eq:s7405})$ and $(\ref{eq:s745})$, or directly from the properties of Fibonacci and Lucas polynomials. In particular, for $p$ a prime number, the $p$\,th Fibonacci and Lucas polynomials are irreducible and so their roots are respectively $2i$ times the real and imaginary parts of the $p$th cyclotomic polynomial, except for the root 0 in the Lucas polynomial case (see for example Koshy \cite{koshy} 2001, p462).
\end{proof}
\begin{example}[of polynomial factorisation and special values]
Taking $m=223=2^5(6+1)-1$, we have
\[
\mathcal{Q}_{223}(x)=\sum_{k=0}^{223}\binom{224+k}{2k+1}x^k=P_3(x)\mathcal{P}_3(x)Q_7(x)Q_{14}(x)Q_{28}(x)Q_{56}(x)Q_{112}(x).
\]
Without proof, some special values of the polynomials $P_m(x)$ and $Q_m(x)$ are
\[
Q_m(-4)=2(-1)^m, ~~P_m(-4)=(-1)^m, ~~Q_m(0)=2, ~~P_m(0)=2m+1,
\]
\[
Q_m(1)=\mathcal{L}_{2m}(1), ~~P_m(1)=\mathcal{L}_{2m+1}, ~~\mathcal{Q}_m(1)=\mathcal{F}_{2m+2}(1), ~~\mathcal{P}_m(1)=\mathcal{F}_{2m+1}.
\]
\end{example}

\section{On Minor Recurrence Relations}
The matrices $M_o(m,r)$ and $M_e(m,r)$, defined after the proof to Theorem 1, respectively generate our sequences $\mathscr{F}_r^{(m,j)}$ and $\mathscr{G}_r^{(m,j)}$, for $1\leq j\leq m$, via the recurrence matrix $R_m$. In consequence, the $i\times i$ minors of these matrices sequence form a set of sequences in their own right, which also obey (different) recurrence relations, for $2\leq i \leq m-1$. The $2\times 2$ minors then correspond to the difference between consecutive convergents $\mathscr{F}_{r+1}^{(m,j)}/\mathscr{F}_{r+1}^{(m,k)} - \mathscr{F}_r^{(m,j)}/\mathscr{F}_r^{(m,k)}$, after multiplying through by the product of the two sequence terms which form the denominators of the convergents.

We now briefly outline the general theories underpinning these $i\times i$ minor recurrence sequence properties.
For $1\leq j \leq m$, let the $m$ sequences $\{y_{j\,k}\}_{k=1}^\infty$ be defined by an $m\times m$ initial value matrix, and an $m$-th order rational linear recurrence matrix $K_m$, of the form (\ref{eq:chr1}) with $a_0=1$, described in the proof of Lemma 3.1. Then $y_{j\,k}$ obeys the recurrence relation
\[
y_{j\,k}=-\left (a_1y_{j\,(k-1)}+a_2 y_{j\,(k-2)}+\ldots +a_m y_{j\,(k-m)}\right ),\quad\text{where}\quad K_m(x)=\sum_{j=0}^m a_{m-j} x^j,
\]
with $K_m(x)$ the characteristic polynomial of $K_m$.

In this general setting we also assume that the system of polynomials $K_m(x)$ are orthogonal, and so satisfy a three-term recurrence, whose measure is supported on some interval $[a,b]\in \mathbb{R}$. This ensures that the roots $\lambda_{m\, 1},\ldots,\lambda_{m\, m}$, of the polynomial equations $K_m(x)=0$, are distinct, real algebraic numbers lying  in the interval $[a,b]$ and that these roots interlace. Hence for $m>n$, there is a root of $K_m(x)=0$ between any two roots of $K_n(x)=0$.

We note that the condition $a_0=1$ produces a system of normalised roots, so that $\lambda_{m\,1}\times \ldots \times \lambda_{m\, m}=1$. We also note that the minimal polynomials for each of the algebraic numbers $\lambda_{m\,j}$ divides the characteristic polynomial $K_m(x)$.

As stated in the proof of Lemma 3.1, the sequences $Y_m^{(j)}=(1,\lambda_{m\,j},\lambda_{m\,j}^2,\lambda_{m\,j}^3,\ldots)$, with $j=1,2,\ldots,m$
form a basis for the solution space for all possible sequences satisfying this recurrence relation, and so for any possible starting values.

Let
$\{y_{1\, k}\}_{k=0}^\infty, \{y_{2\, k}\}_{k=0}^\infty,\ldots, \{y_{i\, k}\}_{k=0}^\infty$ be $i$ sequences generated by an initial value matrix and the linear recurrence, so that in matrix form we can write
\[
Y=\left (
\begin{array}{ccccc}
 y_{1\, 0} & y_{1\, 1} & y_{1\, 2} & y_{1\, 3}& \ldots  \\
  y_{2\, 0} & y_{2\, 1} & y_{2\, 2} & y_{2\, 3} &\ldots   \\
 \vdots & \vdots & \vdots & \vdots & \vdots \\
 y_{i\, 0} & y_{i\, 1} & y_{i\, 2} & y_{i\, 3} &\ldots  \\
\end{array}
\right ),
\]
and consider the sequence formed by successive $i\times i$ determinants
\[
D_\ell=\left |
\begin{array}{ccccc}
 y_{1\, \ell} & y_{1\, \ell+1} & y_{1\, \ell+2} &\ldots & y_{1\, \ell+i-1}  \\
  y_{2\, \ell} & y_{2\, \ell+1} & y_{2\, \ell+2} &\ldots & y_{2\,\ell+i-1}   \\
 \vdots & \vdots & \vdots & \vdots & \vdots \\
 y_{i\, \ell} & y_{i\, \ell+1} & y_{i\,\ell+2} &\ldots & y_{i\, \ell+i-1}   \\
\end{array}
\right |,\quad \ell=0,1,2,3,\ldots
\]
The space of all such sequences $D(Y)=(D_0,D_1,D_2,\ldots)$ is spanned by the $D$'s that you get by choosing an array $Y$ of the form
\[
Y=\left (
\begin{array}{ccccc}
 1 & \gamma_1 & \gamma_1^2 & \gamma_1^3& \ldots  \\
  1 & \gamma_2 & \gamma_2^2 & \gamma_2^3& \ldots  \\
 \vdots & \vdots & \vdots & \vdots & \vdots \\
  1 & \gamma_i & \gamma_i^2 & \gamma_i^3& \ldots  \\
\end{array}
\right ),
\]
where the $\gamma_r$ are $i$ distinct values chosen from the $m$ distinct, real eigenvalues $\lambda_{m\,j}$, of the recurrence matrix. Here the order of the rows of $Y$ is irrelevant as we are simply looking at how to choose an $i$ element subset of an $m$ element set. It follows that there are $\binom{m}{i}$ ways to choose such a $Y$, and from determinant theory the resulting $D(Y)$ is itself a geometric progression of the form $(C,C\Delta,C\Delta^2,\ldots)$, where $\Delta=\gamma_1 \gamma_2 \ldots \gamma_i$. Hence, generically, the space of all such ``$i\times i$ minor sequences'' must be the solution space of a linear constant coefficient recurrence of order at most $\binom{m}{i}$. We have just proved the following lemma.

\begin{lemma}
For $i\geq 1$, the sequence formed by successive $i\times i$ determinants $D_\ell$, as defined above, obeys a linear constant coefficient recurrence of order at most $\binom{m}{i}$. If the eigenvalues of the minor recurrence matrix all have absolute value less than 1, then the sequence of $i\times i$ determinants $D_\ell$ will converge to some number~$\alpha$.
\end{lemma}

\begin{example}[or minor recurrence relation coefficients]
 We illustrate this lemma with the $i$ minor recurrences corresponding to the denominator generating function $P_5(x)$.
\[
\begin{array}{c}\left\{-5,-7,-4,-1,-\frac{1}{11}\right\}\\
\left\{7,-19,\frac{292}{11},-\frac{233}{11},\frac{1223}{121},-\frac{356}{121},\frac{63}{121},-\frac{72}{1331},\frac{4}{1331},-\frac{1}{14641}\right\}\\
\left\{-4,-\frac{72}{11},-\frac{63}{11},-\frac{356}{121},-\frac{1223}{1331},-\frac{233}{1331},-\frac{292}{14641},-\frac{19}{14641},
-\frac{7}{161051},-\frac{1}{1771561}\right \}\\
\left\{1,-\frac{4}{11},\frac{7}{121},-\frac{5}{1331},\frac{1}{14641}\right\}\\
\left\{-\frac{1}{11}\right\}
\end{array}
\]
Here the recurrence coefficients for the sequences $\mathscr{F}_r^{(5,j)}$, are given in the topmost entry, and those for the $2\times 2$ minors the second from top entry, so that they satisfy a recurrence relation with recurrence polynomial
\[
x^{10}-7x^{9}+19x^8-\frac{292}{11}x^7+\frac{233}{11}x^6-\frac{1223}{121}x^5+\frac{356}{121}x^4-\frac{63}{121}x^3+\frac{72}{1331}x^2-\frac{4}{1331}x+
\frac{1}{14641}=0.
\]
The sequences of $m-1\times m-1$ minors appear to
have the recurrence polynomial
\[
\sum_{k=0}^m \left (\frac{(-1)^m}{2m+1}\right )^{k-1} \frac{1}{2m+1-2k}\binom{2m-k}{k}x^{r-k}=0,
\]
so that when $m$ is odd the recurrence polynomial can be factorised as
\[
\prod _{k=1}^m \left(x-\frac{2-2 \cos \left(\frac{2 \pi  k}{2 m+1}\right)}{2 m+1}\right)=0.
\]
In general, the  $\binom{m}{2}+1$ term recurrence relation that our sequences of $2\times 2$ minors obey can be used iteratively to obtain the coefficients
$C_{j\,k}(r)$, such that
\[
\frac{\mathscr{N}_{r+1}^{(m,j)}}{\mathscr{N}_{r+1}^{(m,k)}}-\frac{\mathscr{N}_r^{(m,j)}}{\mathscr{N}_{r}^{(m,k)}}
=
\frac{1}{\mathscr{F}_{r}^{(m,k)}\mathscr{F}_{r+1}^{(m,k)}}\left |
\begin{array}{cc}
\mathscr{F}_{r+1}^{(m,j)} & \mathscr{F}_{r}^{(m,j)}\\
\mathscr{F}_{r+1}^{(m,k)} & \mathscr{F}_{r}^{(m,k)}\\
\end{array}
\right |
\leq \frac{C_{j\,k}(r)}{\left (\mathscr{F}_{r}^{(m,k)}\right )^2},
\]
which is of a similar form to that of Dirichlet's Theorem for standard continued fraction convergents.

\end{example}


\section*{Appendix}
Table of $m$-dimensional integer interlacing Fibonacci sequences $\mathscr{N}_r^{(m,j)}$, for $m=1,2,3,4,5$, with $1\leq j \leq m$, and $1\leq r \leq 10$.
\[
\begin{array}{|c||c||c|c|c|c|c|c|c|c|c|c|}
\hline
{\bf m} & {\bf j/r} & {\bf 1}& {\bf 2} & {\bf 3} & {\bf 4} & {\bf 5} & {\bf 6} &{\bf 7} & {\bf 8} & {\bf 9} & {\bf 10}\\ \hline
{\bf 1}& {\bf 1} &  1 & -1 & 1 & -1 & 1 & -1 & 1 & -1 & 1 & -1  \\
\hline\hline
{\bf 2} & {\bf 1} & 2 & -1 & 3 & -2 & 7 & -5 & 18 & -13 & 47 & -34  \\
 & {\bf 2} & 1 & -1 & 4 & -3 & 11 & -8 & 29 & -21 & 76 & -55 \\
 \hline\hline
{\bf 3} & {\bf 1} & 3 & -2 & 2 & -17 & 22 & -29 & 269 & -357 & 474 & -4406 \\
 & {\bf 2} & 2 & -3 & 4 & -37 & 49 & -65 & 604 & -802 & 1065 & -9900  \\
 & {\bf 3} & 1 & -2 & 3 & -29 & 39 & -52 & 484 & -643 & 854 & -7939  \\
\hline\hline
{\bf 4} & {\bf 1} & 4 & -10 & 46 & -271 & 1702 & -10855 & 69499 & -445420 & 2855494 & -18307378  \\
 & {\bf 2}& 3 & -18 & 108 & -675 & 4293 & -27459 & 175932 & -1127763 & 7230222 & -46355652  \\
 & {\bf 3} & 2 & -17 & 116 & -755 & 4859 & -31184 & 199988 & -1282310 & 8221661 & -52713260  \\
 & {\bf 4} & 1 & -10 & 73 & -487 & 3160 & -20332 & 130492 & -836893 & 5366170 & -34405885  \\
\hline\hline
{\bf 5} & {\bf 1} & 5 & -5 & 11 & -32 & 99 & -3415 & 10744 & -33830 & 106545 & -335575 \\
 & {\bf 2} & 4 & -10 & 28 & -85 & 265 & -9156 & 28817 & -90746 & 285805 & -900180  \\
 &{\bf 3} & 3 & -11 & 35 & -110 & 346 & -11982 & 37734 & -118845 & 374319 & -1178980  \\
 & {\bf 4} & 2 & -9 & 31 & -100 & 317 & -11002 & 34669 & -109210 & 343988 & -1083461  \\
 & {\bf 5} & 1 & -5 & 18 & -59 & 188 & -6535 & 20602 & -64906 & 204447 & -643954 \\
\hline
\end{array}
\]
Table of negative index $m$-dimensional integer interlacing Fibonacci sequences $\mathscr{N}_r^{(m,j)}$, corresponding to the renumbered Fleck number quotients obtained from the sequence terms $\mathscr{F}_r^{(m,j)}=n\,\mathfrak{F}(-2r+1, \,(-r+j)\,\, (\bmod{\,\, n}))$. Values for $m=1,2,3,4,5$, with $1\leq j \leq m$, and $-8\leq r \leq 1$, are given.
\[
\begin{array}{|c||c||c|c|c|c|c|c|c|c|c|c|}
\hline
{\bf m} & {\bf j/r} & {\bf -8}& {\bf -7} & {\bf -6} & {\bf -5} & {\bf -4} & {\bf -3} &{\bf -2} & {\bf -1} & {\bf 0} & {\bf 1}\\ \hline
{\bf 1}& {\bf 1} &  -1 & 1 & -1 & 1 & -1 & 1 & -1 & 1 & -1 & 1 \\
\hline\hline
{\bf 2} & {\bf 1} &  -34 & 47 & -13 & 18 & -5 & 7 & -2 & 3 & -1 & 2 \\
 & {\bf 2} & 21 & -29 & 8 & -11 & 3 & -4 & 1 & -1 & 0 & 1 \\
 \hline\hline
{\bf 3} & {\bf 1} & -493 & 131 & -35 & 66 & -18 & 5 & -10 & 3 & -1 & 3 \\
 & {\bf 2} &  383 & -100 & 26 & -47 & 12 & -3 & 5 & -1 & 0 & 2 \\
 & {\bf 3} & -204 & 52 & -13 & 22 & -5 & 1 & -1 & 0 & 0 & 1 \\
\hline\hline
{\bf 4} & {\bf 1} &   -8103 & 2145 & -572 & 462 & -126 & 35 & -30 & 9 & -3 & 4 \\
 & {\bf 2}&  6477 & -1668 & 429 & -330 & 84 & -21 & 15 & -3 & 0 & 3 \\
 & {\bf 3} &  -4080 & 996 & -238 & 165 & -36 & 7 & -3 & 0 & 0 & 2 \\
 & {\bf 4} & 1836 & -420 & 91 & -54 & 9 & -1 & 0 & 0 & 0 & 1 \\
\hline\hline
{\bf 5} & {\bf 1} &  -2210 & 585 & -156 & 42 & -126 & 35 & -10 & 3 & -1 & 5 \\
 & {\bf 2} &  1768 & -455 & 117 & -30 & 84 & -21 & 5 & -1 & 0 & 4 \\
 &{\bf 3} &-1125 & 273 & -65 & 15 & -36 & 7 & -1 & 0 & 0 & 3 \\
 & {\bf 4} &  561 & -124 & 26 & -5 & 9 & -1 & 0 & 0 & 0 & 2 \\
 & {\bf 5} &-204 & 40 & -7 & 1 & -1 & 0 & 0 & 0 & 0 & 1 \\
\hline
\end{array}
\]

\end{document}